\documentclass{amsart}
\usepackage{amssymb}
\usepackage{bbm}
\usepackage{mathtools}
\usepackage[all,2cell]{xy}
\usepackage{tikz}
\usetikzlibrary{matrix,arrows,decorations.markings, calc, positioning}

\usepackage[colorlinks,citecolor=blue,linkcolor=black]{hyperref}
\usepackage[nameinlink,capitalize]{cleveref}
\DeclareSymbolFont{bbold}{U}{bbold}{m}{n}
\DeclareSymbolFontAlphabet{\mathbbold}{bbold}

\usepackage[shortcuts]{extdash}

\DeclareMathSymbol{\ssquare}{\mathbin}{AMSa}{"03}
\DeclareMathSymbol{\square}{\mathbin}{AMSa}{"03}
\DeclareMathSymbol{\Box}{\mathbin}{AMSa}{"03}

\newcommand{\coTHH}{\mathrm{coTHH}}
\newcommand{\THH}{\mathrm{THH}}
\newcommand{\sma}{\wedge}
\newcommand{\sm}{\wedge}
\newcommand{\coHH}{\mathrm{coHH}}
\newcommand{\HH}{\mathrm{HH}}
\newcommand{\cotor}{\mathrm{Cotor}}
\DeclareMathOperator{\Tot}{Tot}
\DeclareMathOperator{\coker}{coker}

\newcommand{\Spec}{\mathrm{Spec}}
\newcommand{\Set}{\mathrm{Set}}
\newcommand{\sSet}{\mathrm{sSet}}
\newcommand{\cC}{\mathrm{CoCAlg}}
\newcommand{\CAlg}{\mathrm{CAlg}}
\newcommand{\Mod}{\mathrm{Mod}}
\newcommand{\sset}{\sSet}
\newcommand{\CoMod}{\mathrm{BiCoMod}}
\newcommand{\Cat}{\mathrm{Cat}}
\newcommand{\Top}{\mathrm{Top}}

\newcommand{\bbone}{\mathbbold{1}}
\newcommand{\op}{\mathrm{op}}
\newcommand{\Hom}{\mathrm{Hom}}

\newcommand{\inv}{{-1}}
\newcommand{\xto}[1]{\xrightarrow{#1}}
\newcommand{\sh}{\mathrm{sh}}

\newcommand{\AW}{\mathrm{AW}}

\newcommand{\comult}{\triangle}

\newcommand{\Map}{\mathrm{Map}}

\newcommand{\id}{\mathrm{id}}
\newcommand{\boxhopf}[1]{\square_{#1}\textup{-Hopf algebra}}
\newcommand{\boxcoalg}[1]{\square_{#1}\textup{-coalgebra}}

\newcommand{\boxbialg}[1]{\square_{#1}\textup{-bialgebra}}

\newcommand{\iso}{\cong}

\newcommand{\freeloops}{\mathcal{L}}

\theoremstyle{plain}   
\newtheorem{thm}{Theorem}[section] 
\newtheorem{cor}[thm]{Corollary}
\newtheorem{lemma}[thm]{Lemma}
\newtheorem{prop}[thm]{Proposition}
\crefname{prop}{Proposition}{Propositions}
\newtheorem{thrm}[thm]{Theorem}

\theoremstyle{definition}
\newtheorem{defn}[thm]{Definition}
\crefname{defn}{Definition}{Definitions}

\theoremstyle{remark}
\newtheorem{rem}[thm]{Remark}
\newtheorem{ex}[thm]{Example}
\newtheorem{observation}[thm]{Observation}

\makeatletter
\newcommand*{\relrelbarsep}{.386ex}
\newcommand*{\relrelbar}{%
  \mathrel{%
    \mathpalette\@relrelbar\relrelbarsep
  }%
}
\newcommand*{\@relrelbar}[2]{%
  \raise#2\hbox to 0pt{$\m@th#1\relbar$\hss}%
  \lower#2\hbox{$\m@th#1\relbar$}%
}
\providecommand*{\rightrightarrowsfill@}{%
  \arrowfill@\relrelbar\relrelbar\rightrightarrows
}
\providecommand*{\leftleftarrowsfill@}{%
  \arrowfill@\leftleftarrows\relrelbar\relrelbar
}
\providecommand*{\xrightrightarrows}[2][]{%
  \ext@arrow 0359\rightrightarrowsfill@{#1}{#2}%
}
\providecommand*{\xleftleftarrows}[2][]{%
  \ext@arrow 3095\leftleftarrowsfill@{#1}{#2}%
}
\makeatother

\UseAllTwocells

\title[CoTHH and the Homology of Free Loop Spaces]{Topological coHochschild Homology and the Homology of Free Loop Spaces}
\author[Bohmann]{Anna Marie Bohmann}
\author[Gerhardt]{Teena Gerhardt}
\author[Shipley]{Brooke Shipley}

\address{Department of Mathematics, Vanderbilt University, 1326 Stevenson Center, Nashville, TN, 37240, USA}
\email{am.bohmann@vanderbilt.edu}

\address{Department of Mathematics, Michigan State University, 619 Red Cedar Road, East Lansing, MI, 48824, USA}
\email{teena@math.msu.edu}

\address{Department of Mathematics, Statistics, and Computer Science, University of Illinois at
Chicago, 508 SEO m/c 249,
851 S. Morgan Street,
Chicago, IL, 60607-7045, USA}
    \email{shipleyb@uic.edu}

\date{\today}

\makeatletter
\@namedef{subjclassname@2020}{\textup{2020} Mathematics Subject Classification}
\makeatother


\begin{document}

\begin{abstract}
We study the homology of free loop spaces via techniques arising from the theory of topological coHochschild homology (coTHH).  Topological coHochschild homology is a topological analogue of the classical theory of coHochschild homology for coalgebras.  We produce new spectrum-level structure on coTHH of suspension spectra as well as new algebraic structure in the coB\"okstedt spectral sequence for computing coTHH.  These new techniques allow us to compute the homology of free loop spaces in several new cases, extending known calculations.
\end{abstract}
\maketitle

 \section{Introduction}

Given a space $X$, the free loop space $\freeloops X$ is the space of all maps from the circle $S^1$ into $X$. The (co)homology of free loop spaces has been an active area of research for many decades. Interest in this area stems from a number of important applications of free loop spaces to topology and mathematical physics. For example, a classical theorem of Gromoll and Meyer \cite{GrMe69} ties the homology of free loop spaces to the enumeration of closed geodesics on manifolds. The homology of free loop spaces is also the main object of study in the field of string topology \cite{ChasSullivan, CJY}.  

One of the primary strategies for understanding the (co)homology of free loop spaces is to relate it to Hochschild homology \cite{Go85, BuFi86, Men01}.  In this paper, we use the dual construction, coHochschild homology, and the emerging theory of topological coHochschild homology to develop new tools for calculating the homology of free loop spaces. We also describe new algebraic structure on the homology of free loop spaces. 

Recall that topological Hochschild homology (THH) is an analogue for ring spectra of the classical theory of Hochschild homology for algebras. Similarly, the classical theory of coHochschild homology for coalgebras \cite{doi, brzeznski-wisbauer} has a topological analogue for coalgebras in spectra, called topological coHochschild homology (coTHH). The foundations of coTHH have been developed in work of Hess and Shipley \cite{hs.coTHH} as well as work of Bohmann, Gerhardt, H\o genhaven, Shipley, and Ziegenhagen \cite{BGHSZ}. 

Some important examples of coalgebras in spectra are given by suspension spectra of spaces. Indeed, for a space $X$, the diagonal map on $X$ induces a comultiplication 
\[
\triangle\colon \Sigma_+^{\infty}X \to \Sigma_+^{\infty}X \wedge \Sigma_+^{\infty}X. 
\]
which makes $\Sigma^\infty_+X$ into a coalgebra. For $X$ a simply connected space,  $\coTHH( \Sigma_+^{\infty}X)$ can be identified with the suspension spectrum of the free loop space on $X$,
\[
\coTHH( \Sigma_+^{\infty}X) \simeq \Sigma_+^{\infty}\freeloops X.
\] 
Thus tools to study topological coHochschild homology in particular yield information about free loop spaces. 

For topological Hochschild homology, one of the primary computational tools is the B\"okstedt spectral sequence, which relates the homology of THH with the algebraic theory of Hochschild homology. For a field $k$ and a ring spectrum $R$, this has the form
\[
E^2_{*,*} = \HH_*(H_*(R;k)) \Rightarrow H_*(\THH(R);k). 
\]
In prior work, the authors and collaborators constructed an analogous coB\"okstedt spectral sequence computing the homology of coTHH \cite{BGHSZ}. For a field $k$ and a coalgebra spectrum $C$, this spectral sequence has $E_2$-term
\[
E_2^{s,t} = \coHH_{s,t}(H_*(C;k))
\]
and abuts to $H_{t-s}(\coTHH(C);k)$. Here $\coHH$ is the classical coHochschild homology for coalgebras of Doi \cite{doi}. In the case of suspension spectra of simply connected spaces, the coB\"okstedt spectral sequence converges if a Mittag-Leffler condition is satisfied, and gives computational tools to study the homology of free loop spaces:
\[
H_*(\coTHH( \Sigma_+^{\infty}X); k) \cong H_*(\freeloops X; k).
\]

The current paper studies the algebraic structure in the coB\"okstedt spectral sequence, with an eye towards applications to the homology of free loop spaces. Work of Angeltveit and Rognes \cite{AR} shows that under appropriate flatness conditions, the B\"okstedt spectral sequence computing $H_*(\THH(R);k)$  is a spectral sequence of $H_*(R;k)$-Hopf algebras. One might then expect a dual algebraic structure in the coB\"okstedt spectral sequence. This is more subtle than it seems, however, because such an algebraic structure has a multiplication and comultiplication over the coalgebra $H_*(C;k)$, rather than over a ring. In this paper we define the appropriate Hopf algebra-like structure over a coalgebra $D$, which we call a $\Box_D$-Hopf algebra. A $\Box_D$-Hopf algebra $H$ is a $D$-bicomodule with appropriately compatible multiplication and comultiplication maps
\[
\mu\colon H \Box_D H \to H \qquad \text{and} \qquad \triangle\colon H \to H \Box_DH,
\]
where $\Box_D$ denotes the cotensor product over the coalgebra $D$. 

These new constructions allow us to describe new algebraic structure on the homology of free loop spaces.

\begin{thrm}
For $X$ a simply connected space and $k$ a field, if $H_*(\freeloops X;k)$ is coflat as a comodule over $H_*(X;k)$, then $H_*(\freeloops X;k)$ is a $\Box_{H_*(X;k)}$-Hopf algebra. 
\end{thrm}
We obtain this algebraic result as the consequence of spectrum-level structure arising in the $\infty$-category of coalgebra spectra over the suspension spectrum $\Sigma^\infty_+X$, which is discussed in \cref{prop:suspspectrum_boxcoalgspectrallevel}.

We then prove the following result capturing the algebraic structure in the coB\"okstedt spectral sequence:

\begin{thrm}
Let $C$ be a connected cocommutative coalgebra spectrum. The coB\"okstedt spectral sequence is a spectral sequence of  $\Box_{H_*(C;k)}$-coalgebras. Further, if for each $r\geq 2$, $E_r^{*,*}(C)$ is coflat over $H_*(C;k)$, then the coB\"okstedt spectral sequence is a spectral sequence of $\Box_{H_*(C;k)}$-Hopf algebras.
\end{thrm}
This algebraic structure in the coB\"okstedt spectral sequence facilitates new computations of $H_*(\coTHH(C);k)$, and in particular of the homology of free loop spaces $H_*(\freeloops X;k)$.

In this work, we illustrate the power of this structure on the spectral sequence by considering the homology of free loop spaces on simply connected spaces $X$ with certain cohomology rings. We consider in detail the homology of  $\freeloops X$ when $X$ is a simply connected space with exterior cohomology. The (co)homology of $\freeloops X$ in such cases has been considered, for instance, in \cite{KY97}, \cite{Kuribayashi11}, and \cite{KMN14}. Our approach yields new results.
 
 \begin{thm}\label{exteriorloop1}
Let $k$ be a field of characteristic $p$ and let $X$ be a simply connected space whose cohomology is exterior on a finite number of generators
\[
H^*(X; k) \cong  \Lambda_k(x_{i_1}, x_{i_2}, \ldots x_{i_n}),
\]
where the $x_{i_j}$ are generators in odd degrees, $|x_{i_j}|=i_j$, and  $i_{j+1} \geq i_j$. Then when  $\frac{i_n + \sum_{j=1}^n i_j}{i_1-1}\leq p$, the homology of the free loop space on $X$ is given as a graded $k$-module by
\[
H_*(\freeloops X; k) \cong \Lambda_{k}(y_{i_1}, y_{i_2}, \ldots y_{i_n}) \otimes k[w_{i_1}, w_{i_2}, \ldots w_{i_n}],
\]
where $|y_{i_j}| = i_j$, and $|w_{i_j}| = i_j - 1$. 
\end{thm}
The case where $H^*(X; k)$ is exterior on two generators had been previously considered in work of Kuribayashi and Yamaguchi \cite{KY97}. \cref{exteriorloop1} extends their result in the two generator case to a broader range of degrees, while also treating the case of more than two generators.

\subsection{Organization} This paper is organized as follows. In \cref{sec:structures} we recall some classical foundations for coalgebras and introduce the new notion of a $\Box_D$-Hopf algebra over a coalgebra $D$. The definition of coHochschild homology for coalgebras is recalled in  \cref{sec:coHH}, and we prove in that section that the coHochschild homology of a coalgebra $D$ has the structure of a $\Box_D$-bialgebra. In \cref{sec:coTHH} we move to the topological setting, establishing the infinity categorical framework in which to discuss topological coHochschild homology, and defining coTHH. In \cref{sec:freeloop} we use the relationship between coTHH and free loop spaces to show that under coflatness conditions the homology of a free loop space has the structure of a $\Box$-Hopf algebra. In \cref{sect:thespectralsequence} we begin our analysis of the coB\"okstedt spectral sequence. We prove that under coflatness conditions the coB\"okstedt spectral sequence for computing $\coTHH(C)$ is a spectral sequence of $\Box$-Hopf algebras over the homology of $C$. Finally, in \cref{sect:computations} we use this new algebraic structure on the coB\"okstedt spectral sequence to make explicit computations of the homology of free loop spaces.

\subsection{Acknowledgments}  The authors express their gratitude to the organizers of the Women in Topology II Workshop and the Banff International Research Station, where this collaboration began.  We thank Vigleik Angeltveit, Ben Antieau, David Chan, Paul Goerss, Kathryn Hess, Sarah Klanderman, Maximilien P\'eroux, Emily Riehl, and Stephanie Ziegenhagen for helpful conversations related to this work. The authors also thank an anonymous referee for helpful comments. This research was supported by the National Science Foundation [DMS-1710534 and DMS-2104300 to Bohmann, DMS-1810575 to Gerhardt, and DMS-1811278 to Shipley]. Some of this work was done while the second author was in residence at the Mathematical Sciences Research Institute in Berkeley, CA (supported by the National Science Foundation under grant DMS-1440140) during the Spring 2020 semester. The first and third authors would like to thank the Isaac Newton Institute for Mathematical Sciences, Cambridge, for support and hospitality during the program `Homotopy harnessing higher structures' during which some of the work on this paper was carried out. Work during this program was supported by~EPSRC grant no~EP/K032208/1.

\section{(Co)algebraic structures}\label{sec:structures}

Hopf algebra structures arise naturally in the study of (topological) Hochschild homology. For a commutative ring spectrum $R$, $\THH(R)$ is a Hopf algebra in the homotopy category. Further, Angeltveit and Rognes prove that under appropriate flatness conditions, the B\"okstedt spectral sequence computing $H_*(\THH(R);k)$  is a spectral sequence of $H_*(R;k)$-Hopf algebras \cite{AR}. It is then natural to ask whether the coB\"okstedt spectral sequence computing $H_*(\coTHH(C);k)$ for a coalgebra spectrum $C$ is similarly a spectral sequence of Hopf algebras. However, the usual definition of a Hopf algebra over a ring is not the correct framework for this question. The coB\"okstedt spectral sequence should have a multiplication and comultiplication structure over $H_*(C;k)$, but $H_*(C;k)$ is a coalgebra, not a ring. To study the algebraic structure in the coB\"okstedt spectral sequence, one needs a notion of a Hopf algebra-like structure over a coalgebra $D$. In this section we define such a structure, which we call a  $\Box_D$-Hopf algebra.   

We begin by reviewing some standard definitions for coalgebras before arriving at our new constructions of $\Box_D$-algebras, coalgebras, bialgebras, and Hopf algebras. 

\begin{defn}
A \emph{coalgebra} $D$ over a field $k$ is a $k$-vector space along with $k$-linear maps
\[
\triangle\colon D \to D\otimes D \qquad \textup{and} \qquad \epsilon\colon D \to k,
\]
called the comultiplication and counit, such that the following coassociativity and counitality diagrams commute:
\[\xymatrix { D \ar[rr]^\triangle\ar[d]_\triangle& &D\otimes D\ar[d]^{\id\otimes \triangle}\\
D\otimes D\ar[r]^-{\triangle\otimes \id} &(D\otimes D)\otimes D & D\otimes(D\otimes D)\ar[l]_-{\cong}
}
\]
\[\xymatrix{ &D\ar[d]^{\id}\ar[dl]_{\triangle}\ar[dr]^{\triangle}\\
D\otimes D\ar[r]^-{\epsilon\otimes \id} & D & D\otimes D \ar[l]_-{\id\otimes \epsilon}
}\]
The coalgebra $D$ is called \emph{coaugmented} if there is furthermore a coaugmentation $k$-linear map $\eta\colon k \to D$, satisfying the identities
\[
\triangle \eta = \eta \otimes \eta \hspace{1cm} \textup{and} \hspace{1cm} \epsilon \eta = \id. 
\]
\end{defn}

We recall some classical examples of coalgebras over a field which will play an important role in this work. 
\begin{ex}\label{polycoalgex}
Let $D=k[w_1, w_2, \dots]$, where each $w_i$ is in even degree, denote the $k$-coalgebra with vector space basis $\{w_1^{j_1}w_2^{j_2}\dotsm\}_{j_i \geq0}$, and with comultiplication given by
\[
\triangle(w_i^j) = \sum_k  \tbinom{j}{k}\,w_i^k\otimes w_i^{j-k}
\]
on basis elements of the form $w_i^j$.  This defines the comultiplication $\triangle|_i$ on the underlying vector space of $k[w_i]$ for each $i$.  To extend to all basis elements, regard  $w_{i_1}^{j_1}\dotsm w_{i_n}^{j_n}$ as the simple tensor  $w_{i_1}^{j_1}\otimes\dots \otimes  w_{i_n}^{j_n}\in k[w_{i_1}]\otimes\dots \otimes k[w_{i_n}]$.  The comultiplication is defined by the composite
\begin{multline*}
 k[w_{i_1}]\otimes \dots \otimes k[w_{i_n}] \xto{\triangle|_{i_1}\otimes \dots \otimes\triangle|_{i_n}} k[w_{i_1}]\otimes k[w_{i_1}]\otimes \dots \otimes k[w_{i_n}]\otimes k[w_{i_n}]\\
\xto{\ \sigma\ } k[w_{i_1}]\otimes \dots \otimes k[w_{i_n}]\otimes k[w_{i_1}]\otimes \dots \otimes k[w_{i_n}]
\end{multline*}
where $\sigma$ is the evident permutation of the tensor factors.
The counit is given by:
\[
\epsilon(w_1^{j_1}w_2^{j_2}\dotsm) = \begin{cases} 1 &\text{if all } j_i=0\\
0 & \text{if some } j_i>0
\end{cases}
\]
We refer to this as the \emph{polynomial coalgebra.} Some readers will find this example familiar as the underlying coalgebra of the polynomial Hopf algebra; the fact that this coalgebra extends to a Hopf algebra structure determines the comultiplication on all basis elements from its definition on the elements $w_i$. 
\end{ex}

\begin{ex}
Let $D=\Lambda_k(x_1, x_2, \dotsc)$, where the degree of each $x_i$ is odd, denote the $k$-coalgebra with vector space basis $\{x_{i_1}\dotsm x_{i_n}\}_{n\geq 0, i_1<\dots <i_n}$, and with comultiplication on $x_i$ given by
\[
\triangle(x_i) = 1 \otimes x_i + x_i \otimes 1.
\]
The comultiplication can be extended from $\Lambda_k(x_i)$ to all basis elements as above in \cref{polycoalgex}. The counit is given by:
\[
\epsilon(x_{i_1}\dotsm x_{i_n}) = 0, \qquad \epsilon(1) = 1. 
\]
This is the cofree graded cocommutative coaugmented coalgebra on the cogenerators $\{x_1, x_2,\dots\}$  More details about the cofree coalgebra construction may be found in \cite[Chapter 12]{Sweedler}.  We refer to this as the \emph{exterior coalgebra.} Again, this example may be familiar as the underlying coalgebra structure of the exterior Hopf algebra.
\end{ex}

\begin{defn}
A graded $k$-coalgebra $D_*$ is \emph{connected} if $D_*=0$ when $*<0$ and the counit map $\epsilon\colon D_* \to k$ is an isomorphism in degree zero.  
 \end{defn}

\begin{defn}
Let $D$ be a coalgebra over a field $k$.  A \emph{right $D$-comodule} is a $k$-vector space $M$ along with a linear coaction map
\[
\rho_M\colon M \to M \otimes D
\]
such that the following coassociativity and counitality diagrams commute:
\[
\xymatrix{M \ar[r]^{\rho_M} \ar[d]_{\rho_M} & M \otimes D \ar[d]^{\id \otimes \triangle} \\
M \otimes D \ar[r]^-{\rho_M \otimes \id} & M \otimes D \otimes D} \hspace{1cm} 
\xymatrix{M \ar[r]^-{\rho_M} \ar[dr]_{\id} & M\otimes D \ar[d]^{\id\otimes \epsilon} \\ 
& M}
\]
A \emph{left $D$-comodule} is a $k$-vector space $N$ along with a linear coaction map
\[
\rho_N\colon N \to D \otimes N,
\]
such that coassociativity and counitality diagrams analogous to those above commute. 
\end{defn}

\begin{defn}
Let $D$ be a coalgebra over a field $k$. A \emph{$(D, D)$-bicomodule} is a $k$-vector space $M$ that is both a left and right $D$-comodule, such that the following diagram commutes:
\[
\xymatrix{M \ar[r]^-{\rho_M}\ar[d]_{\rho_M'} & M \otimes D \ar[d]^{\rho_M' \otimes \id} \\
D \otimes M \ar[r]^-{\id \otimes \rho_M} & D \otimes M \otimes D}
\]
Here $\rho_M$ denotes the right $D$-coaction map and $\rho_M'$ denotes the left $D$-coaction map. 
\end{defn}

Given a right $D$-comodule $M$ and a left $D$-comodule $N$, one can define their cotensor product $M \Box_D N$ as follows.
\begin{defn}
Let $D$ be a coalgebra over a field $k$. Let $M$ be a right $D$-comodule and $N$ be a left $D$-comodule with coaction maps $\rho_M\colon M \rightarrow M \otimes D$ and $\rho_N\colon N \rightarrow D \otimes N$. The \emph{cotensor product} of $M$ and $N$ over $D$, $M \Box_D N$, is defined as the equalizer in $k$-vector spaces:
\[
M \Box_D N \longrightarrow M\otimes N  \xrightrightarrows[\ \id\otimes \rho_N\ ]{\ \rho_M\otimes \id\ } M \otimes D \otimes N.
\]
\end{defn}

By construction, the cotensor $M \Box_D N$ is naturally a $k$-vector space. If $M$ and $N$ are bicomodules, the resulting cotensor $M \Box_D N$ will again be a bicomodule. Note that this relies on the fact that we are working with coalgebras over a field. For coalgebras over a general commutative ring $R$, the cotensor $M \Box_D N$ may not be a $D$-comodule (cf.\,\cite[11.3]{brzeznski-wisbauer}).  In fact, because we are working over a field, the category  $\CoMod_D$ of bicomodules is an abelian category in which finite products and coproducts are given by direct sum (which agrees with Cartesian product) of $k$-vector spaces \cite[3.26]{brzeznski-wisbauer}. 
The cotensor $\Box_D$ is a monoidal product on this category with unit $D$.

If $B$ and $D$ are cocommutative $k$-coalgebras, a map of cocommutative coalgebras $B\to D$ gives $B$ the structure of a $D$-comodule.  There is an alternate description of $\Box_D$ for such comodules. This is dual to observing that the pushout of commutative algebras produces the tensor product.
\begin{observation}\label{boxprodforcoalg}
Given  cocommutative coalgebras $B_1, B_2,$ and $C$ over a field $k$, along with maps of $k$-coalgebras $f_1\colon B_1 \to C$ and $f_2\colon B_2 \to C$, we can view $B_1$ and $B_2$ as $C$-bicomodules. The right coaction map $\rho_{B_1}$, for instance, is given by
\[
\rho_{B_1}\colon B_1 \xto{\ \triangle\ }  B_1 \otimes B_1 \xto{\ \id \otimes f_1\ }  B_1 \otimes C. 
\]
In this case, the pullback in the category of cocommutative $k$-coalgebras of the diagram  
\[\xymatrix{& B_2\ar[d]^{f_2} \\ B_1\ar[r]^{f_1} & C}
\]
also agrees with the cotensor product $B_1\Box_C B_2$ of $B_1$ and $B_2$ as $C$-bicomodules.  One can readily verify that, since we're working over a field, the cotensor product $B_1\Box_C B_2$ is again a cocommutative coalgebra and satisfies the universal property of the displayed pullback.  
\end{observation}

In our work it will be important to consider the exactness of the cotensor product. 
\begin{defn}
Let $D$ be a coalgebra over a field $k$. A right comodule $M$ over a coalgebra $D$ is \emph{coflat} if $M\square_D-$ is exact as a functor from left $D$-comodules to $k$-vector spaces. 
\end{defn}
For a coalgebra $D$ over a field $k$, any direct summand of a cofree $D$-comodule is coflat.  
In fact, in this case coflat $D$-comodules are precisely the $D$-injective comodules \cite[10.12]{brzeznski-wisbauer}, which are precisely the direct summands of cofree comodules \cite{doi}.

Observe that the standard definitions of coalgebraic structures above all involve working over a base ring, which in our case we take to be the field $k$.  As mentioned, the natural structures arising in $\coTHH$ also include new coalgebraic structures defined with respect to a base coalgebra. In particular, for a coalgebra $D$, we need notions of algebras, coalgebras, bialgebras, and Hopf algebras over $D$.  These definitions require working with $D$-bicomodules.  When $D$ is cocommutative, any $D$-comodule naturally has an induced $D$-bicomodule structure, and the examples we consider in the remainder of the paper are of this form.
\begin{defn}\label{def:boxcoalg}
Let $D$ be a cocommutative coalgebra over a field $k$. A \emph{$\Box_D$\=/coalgebra} $E$ is a $D$-bicomodule together with maps
\[ \triangle\colon E\to E\Box_D E \text{\qquad and\qquad} \epsilon\colon E\to D\]
of $D$-bicomodules satisfying appropriate coassociativity and counitality conditions.  Explicitly, for coassociativity, we require that the following diagram commutes.
\[\xymatrix@C=.7pc@R=1.3pc { E \ar[rr]^-\triangle\ar[d]_-\triangle& &E\Box_D E\ar[d]^-{\id\Box \triangle}\\
E\Box_D E\ar[dr]^-{\triangle\Box \id} && E\Box_D(E\Box_D E)\ar[dl]_-{\cong}\\
& (E\Box_D E)\Box_D E 
}
\]
Counitality is the requirement that the following diagram commutes.
\[\xymatrix{ &E\ar[dl]_{\triangle}\ar[dr]^{\triangle} \ar[d]^{\id}\\
E\Box_D E\ar[r]^-{\epsilon\Box \id} & E & E\Box_D E \ar[l]_-{\id\Box \epsilon}
}\]
A $\boxcoalg{D}$ $E$ is called \emph{coaugmented} if there is furthermore a coaugmentation morphism $\eta\colon D \to E$, satisfying the identities
\[
\triangle \eta = \eta \Box \eta \hspace{1cm} \textup{and} \hspace{1cm} \epsilon \eta = \id. 
\]
\end{defn}

\begin{defn}\label{def:boxalg}
Let $D$ be a cocommutative coalgebra over a field $k$.  A \emph{$\Box_D$\=/algebra} $A$ is a $D$-bicomodule together with maps of $D$-bicomodules $\mu\colon A\Box_D A \to A$ and $\eta\colon D\to A$, satisfying the usual associativity and unitality conditions. A $\Box_D$-algebra $A$ is called \emph{augmented} if furthermore there is an augmentation morphism $\epsilon\colon A \to D$ satisfying the identities
\[
 \epsilon\mu = \epsilon \Box \epsilon \text{\qquad and \qquad} \epsilon \eta = \id. 
\]
\end{defn}

\begin{defn}\label{def:boxbialg}
Let $D$ be a cocommutative coalgebra over a field $k$. A \emph{$\Box_D$\=/bialgebra} $H$ is a $\Box_D$-coalgebra that is also equipped with a multiplication $\mu\colon H\Box_D H\to H$ and a unit $\eta\colon D\to H$ that are maps of $D$-bicomodules. These must satisfy the usual associativity and unit conditions and additionally the following diagrams must commute:
\begin{enumerate}
\item Compatibility of comultiplication and multiplication
\[\xymatrix{H\Box_D H \ar[r]^-{\mu}\ar[d]^-{\triangle\Box\triangle}& H \ar[r]^-{\triangle} &H\Box_D H\\
H\Box_DH\Box_DH\Box_D H\ar[rr]^-{\id\Box\tau\Box \id} && H\Box_DH\Box_D H\Box_D H\ar[u]_-{\mu\Box \mu} 
}
\]
where $\tau$ is the twist.  
\item Compatibility of multiplication and counit
\[\xymatrix{H\Box_D H \ar[dr]^-{\epsilon\Box \epsilon}\ar[rr]^-{\mu} &&H\ar[dl]_{\epsilon}\\
& D}
\]
\item Compatibility of comultiplication and unit
\[\xymatrix{ & D\ar[dl]_{\eta\Box \eta} \ar[dr]^{\eta}\\
H\Box_D H && H\ar[ll]^-
{\triangle}
}\]
\item Compatibility of unit and counit
\[\xymatrix @R=.7pc{D\ar[dr]^{\eta} \ar[dd]^-{\id} \\
&H\ar[dl]^{\epsilon}\\
D
}\]
\end{enumerate}
\end{defn}

\begin{defn}\label{defn:boxhopf} Let $D$ be a cocommutative coalgebra over a field $k$.  A \emph{$\Box_D$\=/Hopf algebra} $H$ is a $\boxbialg{D}$ together with a $D$-bicomodule map $\chi\colon H\to H$ called the \emph{antipode} making the following diagram commute:
\[\xymatrix @R=.7pc{& H\Box_D H \ar[rr]^{\chi\Box \id}&& H\Box_D H\ar[ddr]^{\mu}\\
\\
H\ar[uur]^{\triangle}\ar[rr]^{\epsilon}\ar[ddr]_{\triangle} && D\ar[rr]^{\eta}&& H\\
\\
&H\Box_D H\ar[rr]^{\id\Box \chi}&& H\Box_D H\ar[uur]_{\mu}
}
\]
\end{defn}

\begin{rem}\label{remark:Hopfmonoids}
A $\boxhopf{D}$ can equivalently be defined as a Hopf monoid in the category of $D$-bicomodules.  A \emph{Hopf monoid} in a monoidal category is an object that has compatible  monoid and comonoid structures, together with an antipode.  See \cite[\S 1.2.5]{aguiarmahajan} for more details.
\end{rem}

We now define what it means for elements in a $\Box_D$-Hopf algebra to be indecomposable or primitive. 

\begin{defn}\label{def:indecomposable}
Let $A$ be an augmented $\Box_D$-algebra with augmentation $\epsilon\colon A \to D$.  Let $IA=\ker(\epsilon)$ denote the the augmentation ideal of $A$.  We define the \emph{indecomposable elements} of $A$, denoted $QA$, by the exact sequence
\[
IA \Box_D IA \xto{\ \mu\ }  IA \longrightarrow QA \longrightarrow 0.
\]
In other words, an element is indecomposable if it is in the kernel of the augmentation, but not in the image of the product on the augmentation ideal. 
\end{defn}

\begin{defn}\label{def:primitive}
Let $E$ be a coaugmented $\Box_D$-coalgebra, with coaugmentation $\eta\colon D \to E$. Let $JE$ denote the cokernel of $\eta$, and let $IE$ denote the kernel of the counit $\epsilon\colon E \to D$. Let $PE$ be defined by the exact sequence
\[
0 \longrightarrow  PE \longrightarrow JE \xto{\ \triangle\ } JE \Box_D JE.
\] 
An element $d\in IE$ is \emph{primitive} if its image in $JE$ is in $PE$. 
\end{defn}

\begin{lemma}
Let $E$ be a coaugmented $\Box_D$-coalgebra. If $e\in IE$ then 
\[
\triangle(e) = e \Box_D 1 + 1 \Box_D e + {\textstyle\sum_i} e_{(i)}' \Box_D e_{(i)}''
\]
where $\sum_i e_{(i)}' \Box_D e_{(i)}'' \in IE \Box_D IE$. 
\begin{proof}
The coaugmented $\Box_D$-coalgebra $E$ splits as $E \cong D \oplus IE$, and  
\[
E\Box_DE = (D\Box_D D) \oplus (IE \Box_D D) \oplus (D \Box_D IE) \oplus (IE \Box_D IE).
\]
Then the statement holds because by counitality
\[
\id = (\epsilon \Box_D \id)\circ \triangle = (\id \Box_D \epsilon) \circ\triangle.
\qedhere\] 
\end{proof}
\end{lemma}
Note that the natural map $IE \to JE$ is an isomorphism.  It follows that if $e \in IE$ is primitive, then $\triangle(e) = e \Box_D 1 + 1 \Box_D e$.

The following proposition allows us to understand $\Box_C$-primitive elements in a $\Box_C$-coalgebra of the form $C \otimes D$. This will be computationally useful in \cref{sect:computations}. 

\begin{prop}\label{prop:primitive}
Let $C$ and $D$ be cocommutative coaugmented $k$-coalgebras. 
Then $C\otimes D$ is a $\Box_C$-coalgebra, and an element of the form $c\otimes d \in C \otimes D$ is primitive as an element of the $\Box_C$-coalgebra $C \otimes D$ if and only if $d$ is primitive in the $k$-coalgebra $D$. 
\begin{proof}
Recall that for $C$ and $D$ $k$-coalgebras, the tensor product $C \otimes D$ (over $k$) is also a $k$-coalgebra, with comultiplication given as follows. 
\[
C \otimes D \xto{\ \triangle_C \otimes \triangle_D\ } (C \otimes C) \otimes (D \otimes D) \xto{\ \id \otimes \tau \otimes \id\ } (C\otimes D) \otimes (C \otimes D)
\]
where $\tau$ is the twist.
Note that $C \otimes D$ is a left $C$-comodule, with coaction map
\[
\psi\colon C \otimes D \xto{\ \triangle_C \otimes \id\ } C \otimes C \otimes D.\]
Similarly, $C \otimes D$ is a right $C$-comodule with coaction map
\[
\rho\colon C \otimes D \xto{\ \triangle_C \otimes \id\ }  C \otimes C \otimes D \xto{\ \tau'\ }  C \otimes D \otimes C,
\]
where $\tau'$ rotates the first factor to the end. 
The cotensor $(C\otimes D) \Box_C (C\otimes D)$ is defined as the equalizer
\[
(C\otimes D) \Box_C (C\otimes D) \longrightarrow (C\otimes D) \otimes (C\otimes D) \xrightrightarrows[\id \otimes \psi]{\rho \otimes \id}  (C\otimes D) \otimes C \otimes (C\otimes D).
\]
Observe that the comultiplication on $C\otimes D$ as a $k$-coalgebra induces a comultiplication map $C\otimes D \to (C\otimes D) \Box_C (C\otimes D )$ by the universal property of the equalizer. In particular, the coassociativity of the comultiplication on $C$ guarantees that the diagram below commutes: 
\[
\xymatrix{ & C\otimes D \ar[d]^{\triangle} \ar@{-->}[ld]&  \\(C\otimes D) \Box_C (C\otimes D) \ar[r] & (C\otimes D) \otimes (C\otimes D) \ar@<.5ex>[r]^-{\rho \otimes \id}\ar@<-.5ex>[r]_-{\id \otimes \psi} & (C\otimes D) \otimes C \otimes (C\otimes D).
}
\]

By \cref{def:primitive}, to understand the primitive elements of $C\otimes D$ as a $\Box_C$-coalgebra, we need to first consider the cokernel of the map $\eta\colon C \to C\otimes D$ defined by
\[
C\longrightarrow C \otimes k \xto{\ \id \otimes \eta_D\ } C \otimes D,
\]
where $\eta_D$ is the coaugmentation on $D$. The cokernel of $\eta$, which we will denote $J$, is given by $J = C \otimes \coker(\eta_D)$ because we are working over the field $k$. To identify the primitive elements we must calculate the kernel of $\triangle\colon J \to J\Box_C J$. This map factors as the composite
\[
J \xto{\ \triangle\ }   (C \otimes D) \Box_C (C\otimes D) \xto{}  J \Box_C J.\]

We claim that these primitive elements are $C\otimes PD$, where $PD$ denotes the primitive elements of $D$ as a $k$-coalgebra.  To see this, consider the exact sequence defining the primitive elements of $D$ as a $k$-coalgebra:
\[
0 \longrightarrow PD \longrightarrow \coker(\eta_D) \xto{\ \triangle_D\ }  \coker(\eta_D) \otimes \coker(\eta_D). 
\]
Tensoring this with $C$ over the field $k$ we have
\[
0 \longrightarrow C \otimes PD \longrightarrow C \otimes \coker(\eta_D) \xto{\ \id \otimes \triangle_D\ }  C \otimes( \coker(\eta_D) \otimes \coker(\eta_D)). 
\]
Noting that 
\[
C \otimes( \coker(\eta_D) \otimes \coker(\eta_D)) \cong (C \otimes \coker(\eta_D)) \Box_C (C \otimes \coker(\eta_D)),
\]
the conclusion follows. 
\end{proof}
\end{prop}

\section{coHochschild homology}\label{sec:coHH}

In this section, we recall the definition of coHochschild homology, and prove that under coflatness conditions, the coHochschild homology of $D$ has the structure of a $\Box_D$-bialgebra.

\begin{defn}[{Doi \cite{doi}}]\label{def:coHH1} Let $D$ be a cocommutative coalgebra over a field $k$.  Then the \emph{coHochschild homology} $\coHH(D)$ of $D$ is the homology of the cochain complex $C^*(D)$ defined by
\[ C^n(D)=D^{\otimes n+1}\]
with differential $d$ given by 
\begin{multline*}
d
(d_0\otimes\dots\otimes d_n)= \sum_{i=0}^n (-1)^i d_0\otimes \dots \otimes \triangle(d_i)\otimes \dots \otimes d_n  \\+(-1)^{n+1} \tau_{1,n+1}\big(\triangle(d_0)\otimes d_1\otimes \dots \otimes d_n\big)
\end{multline*}
where $\tau_{1,n+1}$ is the twist map that moves the first tensor factor in $D^{\otimes n+2}$ to the last spot. 

This definition can be extended to define coHochschild homology for graded coalgebras and differential graded coalgebras, as in \cite{hps}. 
\end{defn}

The cochain complex in \cref{def:coHH1} is the (unnormalized) cochain complex of a cosimplicial $k$-module $\coHH^\bullet(D)$ defined as follows. 
\begin{defn}\label{defn:coHH}
Let $D$ be a cocommutative coalgebra over a field $k$. The cosimplicial $k$-module $\coHH^{\bullet}(D)$ is given by
\[
\coHH^{n}(D) = D^{\otimes n+1}
\]
with cofaces
\[\delta_i \colon D^{\otimes n+1} \to D^{\otimes n+2}, \quad \delta_i=\begin{cases}
D^{\otimes i} \otimes \comult \otimes D^{n-i}, & 0\leq i \leq n,\\
\tau(\comult \otimes  D^{\otimes n}), & i=n+1,
\end{cases}\]
where $\tau$ twists the first factor to the last, and codegeneracies
\[\sigma_i \colon D^{\otimes n+2} \to D^{\otimes n+1}, \quad
\sigma_i = D^{\otimes i+1} \otimes \epsilon \otimes D^{\otimes n-i} \quad \text{for } 0 \leq i \leq n.\]
\end{defn}

 Under the generalized Dold--Kan correspondence in the cosimplicial setting, the cosimplicial $k$-module $\coHH^\bullet(D)$ is sent to the normalization of the cochain complex in \cref{def:coHH1}; see \cite[{Corollary/Definition 8.4.3}]{weibel}.

In order to prove that the coHochschild homology of $D$ has the structure of a $\Box_D$-bialgebra, we first verify that some standard results from homological algebra carry over to the setting of $\Box_D$-products. In particular, we will use  versions of the Eilenberg--Zilber Theorem and the K\"unneth Theorem for $\Box_D$-products. 

\begin{prop}\label{prop:EZ} Let $D$ be a coalgebra over a field $k$, and let $A^\bullet$ and $B^\bullet$ be cosimplicial $D$-bicomodules.  Then the Eilenberg--Zilber (shuffle) map induces a quasi-isomorphsim  of cochain complexes
\[\sh \colon C^*(A^\bullet \Box_D B^\bullet) \to C^*(A^\bullet)\Box_D C^*(B^\bullet).\]
The Alexander--Whitney map induces a quasi-inverse map
\[AW\colon N^*(A^\bullet)\Box_D N^*(B^\bullet)\to N^*(A^\bullet \Box_D B^\bullet).\]
\end{prop}
Here the monoidal product $\Box_D$ of $D$-bicomodules extends to a monoidal product on cochain complexes in the usual way: for cochain complexes $A^*$ and $B^*$, the box product $A^*\Box_D B^*$ is defined at level $n$ by
\[(A^*\Box_D B^*)^n=\bigoplus_{p+q=n} A^p\Box_D B^q.\]
In contrast, the monoidal product of cosimplicial $D$-bicomodules is defined levelwise.

\begin{proof}  After applying simplicial/cosimplicial duality, this result follows from the generalized Eilenberg--Zilber theorem for an arbitrary abelian category, which is proved in \cite[Section 8.5]{weibel}.  In more detail, the category $\CoMod_D^\Delta$ of cosimplicial $D$-bicomodules is equal to the opposite category of \emph{simplicial} objects in $\CoMod_D^\op$, and the category of nonnegatively graded cochain complexes in $\CoMod_D$ is equal to the opposite category of nonnegatively graded \emph{chain} complexes in $\CoMod_D^\op$. The generalized Eilenberg--Zilber theorem applies to the abelian category $\CoMod_D^\op$ to produce natural shuffle and Alexander--Whitney maps in chain complexes in $\CoMod_D^\op$, which then dualize to the maps above. 
\end{proof}

\begin{prop}[{K\"unneth Theorem over $\Box_D$}]\label{thrm:boxkunneth}
 Let $D$ be a coalgebra over a field $k$ and let $A^*$ and $B^*$ be bounded below cochain complexes of $D$-bicomodules.  Then there is a K\"unneth map 
\[ H(A^*\Box_D B^*) \to H(A^*)\Box_D H(B^*).\]
If $A^*$ is a complex of coflat $D$-comodules and all $H_i(A^*)$ (or $H_i(B^*)$) are coflat $D$-comodules, then the K\"unneth map is an isomorphism. 
\end{prop}
\begin{proof}
 The naturality of the  K\"unneth theorem for cochain complexes of  $k$-modules implies that we get commutative squares choosing either the left or right vertical arrows in the lower square below:  
\[
\xymatrix{
H(A^*\Box_D B^*)\ar[d]\ar@{-->}[r] & H(A^*)\Box_D H(B^*)\ar[d]\\
H(A^*\otimes_k B^*)\ar@<-1ex>[d]\ar@<1ex>[d] \ar[r]^-\cong& H(A^*)\otimes_k H(B^*)\ar@<-1ex>[d]\ar@<1ex>[d]\\
H(A^*\otimes_k D\otimes_k B^*) \ar[r]^-\cong & H(A^*)\otimes_k D\otimes_k H(B^*)
}
\]
Here, to obtain the parallel maps in the left column, we view $D$ as a cochain complex concentrated in degree $0$ and note that the right coaction maps $A^n\to A^n\otimes D$ assemble into a map of $k$-cochain complexes $A^*\to A^*\otimes D$; similarly for the left coaction maps on $B^*$. The $\Box_D$ product $A^*\Box_D B^*$ is the equalizer in cochain complexes of these coaction maps. 

Since the the top left arrow equalizes the lower left parallel arrows in the category of cochain complexes and since homology is functorial, the top left arrow equalizes the right parallel arrows after passing through the horizontal isomorphisms.  Hence there is an induced map to the equalizer of the right lower parallel arrows, which is $H(A^*)\Box_D H(B^*)$ as shown.  
This is the desired map
\[H(A^*\Box_D B^*)\to H(A^*)\Box_D H(B^*).\]
If $A^*$ is a complex of coflat $D$-comodules, the dual K\"unneth spectral sequence \cite{EilenbergMoore66} has the form
\[
E_2^{p,q} = \bigoplus_{s+t=q} \cotor_{D}^p(H_s(A^*), H_t(B^*)) \Rightarrow H_{p+q}(A^*\Box_D B^*).
\]
If all $H_i(A^*)$ (or all $H_i(B^*)$) are coflat $D$-comodules, $ \cotor_{D}^p(H_s(A^*), H_t(B^*))$ is trivial for $p>0$, and this spectral sequence collapses, yielding the desired isomorphism. 
 \end{proof}

Having established these homological algebra results, we now verify that for a cocommutative coalgebra $D$ over a field, its coHochschild homology, $\coHH(D)$, is a $\Box_D$-bialgebra. Note that the cosimplicial $k$-module $\coHH^\bullet(D)$ is a cosimplicial $D$-bicomodule.  At level $n$, the left coaction of $D$ is given by the comultiplication on the first copy of $D$:
\[ D\otimes (D^{\otimes n})\xto{\triangle\otimes \id} D\otimes D\otimes (D^{\otimes n}).\]
The right coaction is given by comultiplication followed by a twist:
\[ D\otimes (D^{\otimes n})\xto {\triangle \otimes \id} D\otimes D \otimes (D^{\otimes n})\xto{\tau_{1,n+1}} D\otimes (D^{\otimes n})\otimes D.\]

\begin{prop}\label{prop:coHHcoalg}
Let $D$ be a cocommutative coalgebra over a field $k$. Then the coHochschild homology, $\coHH_*(D)$, is a $\Box_D$-coalgebra. 
\end{prop}

\begin{proof}
We must define a map of $D$-comodules $\coHH_*(D)\to \coHH_*(D)\Box_D \coHH_*(D)$. We use the generalized Eilenberg--Zilber theorem to first build this map at the cosimplicial level.

  At cosimplicial level $n$, the cocommutative comultiplication $\triangle\colon D\to D\otimes D$ induces a map
\[\triangle \colon \coHH^n(D)\to \coHH^n(D)\otimes \coHH^n(D)\] 
by applying $\triangle$ to each tensor factor and then reordering to interleave the copies of $D$. Because $D$ is coassociative and cocommutative, this map equalizes the left and right $D$-comodule action at each level and thus produces a map of $D$-comodules
\[\coHH^n(D) \to \coHH^n(D)\Box_D\coHH^n(D).\]
Hence the map of cosimplicial $k$-modules
\[ \triangle^\bullet\colon\coHH^\bullet(D)\to \coHH^\bullet(D)\otimes \coHH^\bullet(D)\]
in fact induces a map to the equalizer
\[\triangle^\bullet\colon\coHH^\bullet(D)\to \coHH^\bullet(D)\Box_D\coHH^\bullet(D),\]
recalling that limits (and colimits) in cosimplicial objects are computed levelwise.

Applying the Dold--Kan equivalence, we obtain a map of cochain complexes
\[ N^*(\coHH^\bullet(D))\to N^*(\coHH^\bullet(D)\Box_D\coHH^\bullet).\]
The generalized cosimplicial Eilenberg--Zilber theorem of \cref{prop:EZ} implies that the shuffle map induces a quasi-isomorphism of cochain complexes of $D$-bicomodules
\[N^*(\coHH^\bullet(D)\Box_D\coHH^\bullet(D))\to N^*(\coHH^\bullet(D))\Box_DN^*(\coHH^\bullet(D)).\]
Composing with the map induced from $\triangle^{\bullet}$ above induces a map on homology
\[ \coHH_*(D)\to H_*(N^*(\coHH^\bullet(D))\Box_D N^*(\coHH^\bullet(D))).\] 
By \cref{thrm:boxkunneth} above, we then have a map
\[
\coHH_*(D)\to \coHH_*(D)\Box_D \coHH_*(D).
\]
The counit $\epsilon\colon \coHH_*(D) \to D$ is the map that sends $\coHH_*(D)$ to the $0$-cochains. One can check coassociativity and counitality by hand, or we will see later in \cref{coHHbialg2} that these conditions follow from an alternate simplicial description of this structure.
\end{proof}

\begin{prop}\label{coHHbialg} If $\coHH_*(D)$ is coflat over $D$, then $\coHH_*(D)$ is additionally a $\Box_D$-algebra.
\end{prop}
\begin{proof}
There is a cochain level product
\[
\mu\colon C^*(D) \Box_D C^*(D) \to C^*(D),
\]
where $C^*(D)$ denotes the cochain complex computing $\coHH_*(D)$. Let 
\[
{\textstyle\sum}(d_0 \otimes d_1 \otimes \ldots d_i) \otimes(d'_0 \otimes d'_1 \otimes \ldots d'_j)  \in C^*(D) \Box_D C^*(D)
\]
The product $\mu$ is given on each summand 
\[
(d_0 \otimes d_1 \otimes \ldots d_i) \otimes (d'_0 \otimes d'_1 \otimes \ldots d'_j)
\]
by the composite
\[
D^{\otimes (i+1)} \otimes D^{\otimes (j+1)} \xto{\ \id^{i+1} \otimes \epsilon \otimes \id^{j}\ }  D^{\otimes (i+1)} \otimes k \otimes D^{\otimes (j)} \xto{\ \cong\ }  D^{\otimes (i+j+1)},
\]
where $\epsilon$ is the counit in $D$. When $\coHH_*(D)$ is coflat as a $D$-comodule, by \cref{thrm:boxkunneth} this induces a product on coHochschild homology: 
\[
\mu\colon \coHH_*(D) \Box_D \coHH_*(D) \to \coHH_*(D). 
\]
There is also a unit map $\eta\colon D \to \coHH_*(D)$ given by the inclusion of the $0$-cochains.
It is clear from the definition on cochains that the product $\mu$ is associative and unital.  
\end{proof}

We show below that the $\Box_D$-coalgebra structure of \cref{prop:coHHcoalg} and the $\Box_D$-algebra structure of \cref{coHHbialg} will make $\coHH_*(D)$ a $\Box_D$-bialgebra when $\coHH_*(D)$ is coflat over $D$. Before proving this, we consider an alternate description of the $\Box_D$-algebra and coalgebra structures from a cosimplicial perspective.

The identification of the coHochschild cochain complex as arising from the  cosimplicial coalgebra of \cref{defn:coHH} allows us to identify much of the structure on coHochschild homology as arising at this cosimplicial level.  In particular, $\coHH^\bullet(D)$ can be viewed as a cosimplicial cotensor of $S^1_\bullet$ with the coalgebra $D$.   Similar structure exists for Cartesian monoidal $\infty$-categories, as we discuss in the next section.  More precisely, for a cocommutative graded $k$-coalgebra over a field $k$, the cosimplicial $k$-module of \cref{defn:coHH} has the structure of a cosimplicial $k$-coalgebra.  By inspection, this structure is in fact given by viewing $\coTHH_\bullet(D)$ as the cosimplicial cotensor object $D^{S^1_\bullet}$.

With this identification, the $\Box_D$-coalgebra structure on $\coHH_*(D)$, as defined in \cref{prop:coHHcoalg}, is induced by the simplicial fold map $\nabla\colon S^1_{\bullet} \vee S^1_{\bullet} \to S^1_{\bullet}$. The $\Box_D$-product structure on $\coHH(D)$, as in \cref{coHHbialg}, is induced by a simplicial pinch map. In order to arrive at a pinch map that is indeed simplicial, however, one must use a different simplicial model of the circle. In this work we will use two different ``double circle'' models, denoted $dS^1_{\bullet}$ and $d'S^1_\bullet$ respectively (following \cite{AR}). The double circle $dS^1_\bullet$ is given by 
 \[ dS^1_\bullet =(\Delta^1\amalg \Delta^1)\amalg_{(\partial\Delta^1\amalg\partial\Delta^1)}\partial\Delta^1.\]   
  The double circle $d'S^1_{\bullet}$ is given by the quotient of the double 1-simplex $d\Delta^1 = \Delta^1 \amalg_{\Delta^0}\Delta^1$ by its two end-points, $\partial d\Delta^1$. It is this double circle model $d'S^1_{\bullet}$ that is relevant in the current section. We revisit the model $dS^1_{\bullet}$ in \cref{sec:freeloop,sect:thespectralsequence}.

There is a simplicial pinch map
 \[
 \psi\colon d'S^1_{\bullet} \to S^1_{\bullet} \vee S^1_{\bullet}.
 \]
 Let $d'\coHH^\bullet(D)$ denote the cosimplicial $k$-coalgebra $D^{d'S^1_\bullet}$.  To parallel \cref{def:coHH1}, we let $d'\coHH_*(D)$ denote the homology of the cochain complex of $k$-modules $C^*(d'\coHH^\bullet(D))$ associated to $d'\coHH^\bullet(D)$ under the Dold--Kan correspondence.  Then the simplicial pinch map $\psi$ induces a map
\[
\coHH_*(D)\Box_{D} \coHH_*(D) \to d'\coHH_*(D).
\]
We observe that the map $\pi\colon d'S^1_{\bullet} \to S^1_{\bullet}$, given by collapsing the second $\Delta^1$ to a point, induces an isomorphism 
\[
\coHH_*(D) \to d'\coHH_*(D).
\]
This can be proven by dualizing the proof in \cite[Lemma 2.2]{AR}. 
Thus, the simplicial pinch map $\psi$ induces a $\Box_D$-product
\[
\coHH_*(D)\Box_{D} \coHH_*(D) \to \coHH_*(D).
\]
We would like to see that this product agrees with the cochain level product from \cref{coHHbialg}.

\begin{prop}\label{prop:simplicialbialg}
If $\coHH_*(D)$ is coflat over $D$, the $\Box_D$-product 
\[\coHH_*(D)\Box_{D} \coHH_*(D) \to \coHH_*(D)\]
 induced by the simplicial pinch map $\psi\colon d'S^1_{\bullet} \to S^1_{\bullet} \vee S^1_{\bullet}$ agrees with the cochain level $\Box_D$-product defined in \cref{coHHbialg}. 
\end{prop}
\begin{proof}
In work of Angeltveit and Rognes \cite{AR} they prove a dual result. In the dual case, one views Hochschild homology $\HH_{\bullet}(A)$ as the simplicial tensor with $S^1_{\bullet}$, $\HH_{\bullet}(A) = A \otimes S^1_{\bullet}$.  Angeltveit and Rognes then show that the simplicial pinch map induces a comultiplication on Hochschild homology
\[
\HH_*(A) \to \HH_*(A) \otimes_{A} \HH_*(A), 
\]
and verify that when $\HH_*(A)$ is flat over $A$, this comultiplication agrees with the one induced by the chain level comultiplication, $\psi\colon C_*(A) \to C_*(A) \otimes_A C_*(A)$ given by
\[
\psi(a_0 \otimes a_1 \otimes \cdots \otimes a_q) = \sum_{i=0}^q (a_0 \otimes a_1 \otimes \cdots \otimes a_i) \otimes_A (1 \otimes a_{i+1} \otimes \cdots \otimes a_q).
\] 
Their identification uses the shuffle equivalence 
\[
\sh\colon C_*(A) \otimes_A C_*(A) \to \textup{Ch}(A \otimes (S^1_\bullet \vee S^1_\bullet)).\]
The dual shuffle equivalence that we need,
\[
\sh\colon \textup{CoCh}(D^{S^1_\bullet\vee S^1_\bullet}) \to  C^*(D) \Box_D C^*(D)
\]
is defined in our generalized Eilenberg-Zilber theorem, \cref{prop:EZ}. Here $C^*(D)$ denotes the cochain complex for $\coHH_*(D)$, as in \cref{def:coHH1}. The statement in the proposition then follows by dualizing the proof of \cite[Proposition 2.3]{AR}.
\end{proof}

\begin{prop} If $\coHH_*(D)$ is coflat over $D$, then $\coHH_*(D)$ is a $\Box_D$-bialgebra.\label{coHHbialg2}
\end{prop}
\begin{proof}
The multiplication, comultiplication, unit, and counit maps are induced by simplicial maps of the circle. In particular, there are simplicial maps: 
\begin{alignat*}{2}
\eta\colon& S^1_{\bullet} \to \ast & \epsilon\colon& \ast \to S^1_{\bullet} \\
\nabla\colon& S^1_{\bullet} \vee S^1_{\bullet} \to S^1_{\bullet}\qquad &
\psi\colon& d'S^1_{\bullet} \to S^1_{\bullet} \vee S^1_{\bullet},
\end{alignat*}
given by retraction to the basepoint, inclusion of the basepoint, fold, and pinch. Using the identification of $\coHH^{\bullet}(D)$ as the cosimplicial cotensor of $D$ with $S^1_{\bullet}$, these simplicial maps induce the following maps on $\coHH_*(D)$:
\begin{alignat*}{2}
\eta\colon& D \to \coHH_*(D) \\
 \epsilon\colon& \coHH_*(D) \to D \\
\triangle\colon& \coHH_*(D) \to \coHH_*(D) \Box_D \coHH_*(D)\\
\mu\colon& \coHH_*(D) \Box_D \coHH_*(D) \to \coHH_*(D),
\end{alignat*}
By \cref{prop:simplicialbialg} and the preceding discussion, these maps agree with the maps described concretely in \cref{prop:coHHcoalg,coHHbialg}. It remains to check that the appropriate diagrams commute. The coassociativity and counitality of the $\Box_D$-comultiplication $\triangle$ follow from the commutativity of the diagrams of simplicial sets below:
\[
\xymatrix{ S^1_{\bullet} & & S^1_{\bullet} \vee S^1_{\bullet} \ar[ll]_{\nabla} \\
S^1_{\bullet} \vee S^1_{\bullet} \ar[u]^{\nabla}& (S^1_{\bullet} \vee S^1_{\bullet})\vee S^1_{\bullet} \ar[l]_-{\nabla \vee \id} \ar[r]^{\cong} & S^1_{\bullet} \vee(S^1_{\bullet} \vee S^1_{\bullet}) \ar[u]_{\id \vee \nabla}
} 
\]
and
\[
\xymatrix{ & S^1_{\bullet} & \\
 S^1_{\bullet} \vee S^1_{\bullet} \ar[ur]^{\nabla} & S^1_{\bullet} \ar[l]_-{\epsilon \vee \id} \ar[u]_{\id} \ar[r]^-{\id \vee \epsilon} & S^1_{\bullet} \vee S^1_{\bullet} \ar[ul]_{\nabla}
}
\]
It is straightforward to verify that the product $\mu$ is associative and unital using the cochain level formulas in \cref{coHHbialg}. One can also show this using diagrams of simplicial sets, although it requires a triple model for the circle $S^1$. We omit these details here. 

Finally, to conclude that $\coHH_*(D)$ is a bialgebra, we need some compatibility between the $\Box_D$-algebra structure and the $\Box_D$-coalgebra structure, as described in \cref{def:boxbialg}. The compatibility of the multiplication and comultiplication follows from the following commutative diagram of simplicial sets:
\[\xymatrix{S^1_{\bullet}\vee S^1_{\bullet} & d'S^1_{\bullet} \ar[l]_{\psi} & d'S^1_{\bullet} \vee d'S^1_{\bullet} \ar[l]_{\nabla} \ar[d]^{\psi \vee \psi}\\
S^1_{\bullet}\vee S^1_{\bullet} \vee S^1_{\bullet}\vee S^1_{\bullet} \ar[u]^{\nabla \vee \nabla} & & S^1_{\bullet}\vee S^1_{\bullet} \vee S^1_{\bullet}\vee S^1_{\bullet} \ar[ll]_{\id \vee \tau \vee \id}
}
\]
where $\tau$ is a twist map that swaps two factors. Compatibility of the multiplication and counit, compatibility of the comultiplication and unit, and compatibility of the unit and counit also follow directly from commutative diagrams of simplicial sets. Checking these compatibility conditions is left to the reader. 
\end{proof}

As an example, we compute the $\Box_{\Lambda_k(x)}$-bialgebra structure on $\coHH_*(\Lambda_k(x))$. This example will play a key role in later computations.

\begin{prop}\label{prop:exteriorbialg}
Consider the exterior coalgebra $\Lambda_k(x)$, where $|x|=2n+1$. There is a $\Box_{\Lambda_k(x)}$-coalgebra isomorphism
\[
\coHH_*(\Lambda_k(x)) \cong \Lambda_k(x) \otimes  k[\sigma x],
\]
where $(\sigma x)^q$ is the cohomology class of the cocycle $1 \otimes x \otimes x \dots \otimes x \in C^q(\Lambda_k(x)).$ The $\Box_{\Lambda_k(x)}$-multiplication on classes $(\sigma x)^q$ is given by $\mu((\sigma x)^q \otimes (\sigma x)^s) = (\sigma x)^{q+s}$. 
\end{prop}
\begin{proof}
Let $D$ denote the exterior coalgebra $\Lambda_k(x)$. Let $\overline{D}$ denote the kernel of the counit $\epsilon\colon D\to k$.  Recall that the normalized coHochschild cochain complex $(NC^*(D), d)$ is given in degree $q$ by:
\[
NC^q(D) = \bigcap_{i=0}^{q-1} \ker(\sigma_i) \cong D \otimes \overline{D}^{\otimes q}. 
\]
Hence,
\[
NC^q(D) \cong D\{1\otimes x^{\otimes q}\}.
\]
We now compute the coboundary map $d$ on the normalized cochain complex. Note that the coface maps $\delta_i$ for $0<i<q+1$ are zero because 
\[
\triangle(x) = 1\otimes x + x \otimes 1 = 0 \in \overline{D} \otimes \overline{D}.
\]  
Similarly,
\[
\delta_0(1\otimes x^{\otimes q}) = \delta_{q+1}(1\otimes x^{\otimes q}) =0 \in D \otimes \overline{D}^{\otimes q}.
\]
We now consider $\delta_{q+1}(x \otimes x^{\otimes q}).$ By definition,
\[
\delta_{q+1}=     \tau\circ (\triangle \otimes \id^{\otimes q}),
\]
where $\tau$ is the rotating isomorphism that brings the front factor to the end. Note that in this graded setting, $\tau$ incorporates a sign:
\[
\tau(D_{i_0} \otimes D_{i_1} \otimes \dots\otimes D_{i_{q+1}}) = (-1)^{i_0(i_1 + i_2 + \dots+ i_{q+1})} (D_{i_1} \otimes D_{i_2} \otimes \dots \otimes D_{i_{q+1}}\otimes D_{i_0}).
\]
Thus, $\delta_{q+1}(x\otimes x^{\otimes q}) = (-1)^q(1 \otimes x^{q+1})$ in $D \otimes \overline{D}^{\otimes (q+1)},$ and hence the sum 
\[\delta_0(x\otimes x^{\otimes q}) + (-1)^{q+1}\delta_{q+1}(x\otimes x^{\otimes q})\] is zero. From these calculations of coface maps we conclude that $d=0$ on $NC^q(D)$, so 
\[
\coHH_q(\Lambda_k(x)) \cong \Lambda_k(x)\{1 \otimes x^{\otimes q}\}.
\] 
We write $\sigma x$ for the class $1 \otimes x$ and $(\sigma x)^q$ for the class $1 \otimes x^{\otimes q}$. From the description of the comultiplication in \cref{prop:coHHcoalg} it then follows that there is an isomorphism of $\Box_{\Lambda_k(x)}$-coalgebras 
\[
\coHH_*(\Lambda_k(x)) \cong \Lambda_k(x) \otimes  k[\sigma x].
\]
The formula for multiplication on the classes  $(\sigma x)^q$ follows directly from the formulas for $\Box$-multiplication in \cref{coHHbialg}. 
\end{proof}

\section{Coalgebra spectra and topological coHochschild homology}\label{sec:coTHH}

In this section we establish the infinity categorical framework in which to discuss cocommutative coalgebra spectra and cocommutative $Hk$-coalgebras.  Recall that there is a symmetrical monoidal $\infty$-category of spectra, which we  denote $\Spec$.  We denote the symmetric monoidal $\infty$-category of $Hk$-modules by $\Mod_{Hk}$. As dual objects to algebras, coalgebras are then defined to be algebras in the opposite category.

\begin{defn}[{An instance of \cite[Definition 3.1.1]{LurieElliptic}}]\label{defncoalspectra}
A \emph{cocommutative coalgebra spectrum} $C$ is a commutative algebra in the opposite $\infty$-category $\Spec^\op$.  The $\infty$-category of cocommutative coalgebra spectra is defined to be $(\CAlg(\Spec^\op))^\op$, that is, the opposite  category of the category of commutative algebras in $\Spec^\op$. We denote this $\infty$-category by $\cC(\Spec)$. 
\end{defn}

Similarly, there is an $\infty$-category of cocommutative $Hk$-coalgebras. 
\begin{defn}[{\cite[Definition 3.1.1]{LurieElliptic}}]\label{defnHkcoalgspectra}
A \emph{cocommutative $Hk$-coalgebra} is a commutative algebra in the opposite $\infty$-category $\Mod_{Hk}^{\op}$. The $\infty$-category of cocommutative $Hk$-coalgebras is defined to be $(\CAlg(\Mod_{Hk}^\op))^\op$. We denote this $\infty$-category by $\cC(\Mod_{Hk}).$
\end{defn}

\begin{rem} Note that in a classical category $\mathcal{C}$, the opposite category of commutative algebras in $\mathcal{C}^\op$ is the category of cocommutative coalgebras in $\mathcal{C}$.
\end{rem}

In \cite[Proposition 3.1.4]{LurieElliptic}, Lurie shows that under presentability conditions on $\mathcal{C}$ and accessibility conditions on the monoidal product, the $\infty$-category of cocommutative coalgebra objects in $\mathcal{C}$ is presentable.  This applies to the $\infty$-categories of cocommutative $Hk$-coalgebra spectra and of cocommutative coalgebra spectra.  In particular, these categories have all small limits and colimits, and colimits of cocommutative ($Hk$)-coalgebras are preserved by the forgetful functor from cocommutative ($Hk$)-coalgebra spectra to spectra.  

However, limits of cocommutative coalgebra spectra in general do \emph{not} agree with limits in underlying spectra.  This is essentially because the smash product of spectra doesn't play nice with limits, in contrast to colimits. 

Further results of Lurie's \emph{Higher Algebra} show that the $\infty$-categories of coalgebra spectra of \cref{defncoalspectra,defnHkcoalgspectra} are symmetric monoidal.
\begin{prop}
The $\infty$-categories $\cC(\Spec)$ and $\cC(\Mod_{Hk})$ have Cartesian symmetric monoidal structures under which the (finite) Cartesian product is given on objects by the smash product of spectra.
\end{prop}

\begin{proof}
As observed in \cite[Remark 2.4.2.7]{lurieHA}, if $\mathcal{C}$ is a symmetric monoidal $\infty$-category, there is an induced symmetric monoidal structure on $\mathcal{C}^\op$.  Just as in the classical case, as discussed in \cite[\S 2.4.3]{lurieHA}, this duality takes Cartesian symmetric monoidal structures to coCartesian symmetric monoidal  structures.  Hence, it suffices to prove that $\CAlg(\Spec^\op)$ and $\CAlg(\Mod_{Hk}^\op)$ have the structure of coCartesian symmetric monoidal $\infty$-categories.  This result, in turn, is \cite[Proposition 3.2.4.7]{lurieHA}, which shows that for commutative algebras, the underlying symmetric monoidal product is the coproduct.
\end{proof}

Given a cocommutative coalgebra $C\in \cC(\Spec)$, one might ask for a nice $\infty$-category of comodules over $C$, in which to define the analogs of $\Box_C$-coalgebras and the related algebraic structures of  \cref{sec:structures}. While the dual construction of $\infty$-categories of modules over algebras has been extensively developed by Lurie \cite{lurieHA}, these constructions do not readily dualize because the smash product in spectra does not commute with totalization.  Hence we do not generally have a suitable symmetric monoidal $\infty$-category of $C$-comodules. In later sections, we largely avoid this point by working in the special case of $Hk$-coalgebras, where the Dold--Kan correspondence allows us to transfer questions to the realm of (co)chain complexes. This setting is also the subject of P\'eroux's rigidification results \cite{perouxrigid}.

 For a general cocommutative coalgebra $C$, we can also recover an analog of the structure of a $\Box_C$-coalgebra via the viewpoint of \cref{boxprodforcoalg}, which remarks that the box-product $B_1\Box_C B_2$ of coalgebras over $C$ can be defined via pullback over $C$.  
\begin{prop}\label{inftyboxprodstructurecoalg}
Let $C\in\cC(\Spec)$.  Then the over category  $\cC(\Spec)_{/C}$  is a Cartesian monoidal $\infty$-category where the monoidal product $B_1\Box_C B_2$ of objects $B_1\to C$ and $B_2\to C$ is the pullback in $\cC(\Spec)$:
\[ \xymatrix{ B_1\Box_C B_2 \ar[r]\ar[d] & B_2\ar[d]\\
B_1\ar[r] & C.}\]
Hence every object of $\cC(\Spec)_{/C}$ is a cocommutative $\Box_C$-coalgebra spectrum.
\end{prop}

\begin{proof}
For concreteness, we use Joyal's definition of the $\infty$-category $\cC(\Spec)_{/C}$ of cocommutative coalgebra spectra over $C$ \cite{Joyal02}. That is, we view $C$ as a diagram $C\colon \bbone\to \cC(\Spec)$, where $\bbone$ is the category with one object and only the identity morphism.  Then  $\cC(\Spec)_{/C}$ is the $\infty$-category defined by the property that for a simplicial set $Y$, 
\[ \Map(Y,\cC(\Spec)_{/C})\simeq \Map_{\bbone}(Y\star \bbone, \cC(\Spec)).\]
The subscript on the right-hand side indicates that we consider only those maps $f\colon Y \star \bbone \to \cC(\Spec)$ such that $f|_{\bbone}$ is $C$. The $\infty$-category $Y\star \bbone$ is the join of the simplicial sets $Y$ and $\bbone$, as defined in \cite{Joyal02}.  Since $\bbone$ is a single point, this operation produces the ``cocone'' on an $\infty$-category $Y$. 

Since the category $\cC(\Spec)$ is defined as the opposite of $\CAlg(\Spec^\op)$, by passing to opposite categories, it suffices to show that the under category $\CAlg(\Spec^\op)_{C^\op/}$ admits the structure of a coCartesian monoidal $\infty$-category.  More precisely, for an arbitrary simplicial set $Y$, the join/slice adjunction \cite[Proprosition 4.2.5]{RVelements} and op-duality provide a string of equivalences 
\begin{align*}
\Map(Y,\cC(\Spec)_{/C}) &\simeq \Map_{\bbone}(Y\star \bbone, \cC(\Spec))\\
&\simeq \Map_{\bbone}(\bbone\star (Y^\op), \CAlg(\Spec^\op))\\
&\simeq \Map(Y^\op, \CAlg(\Spec^\op)_{C^\op/})\\
&\simeq \Map(Y, (\CAlg(\Spec^\op)_{C^\op/})^\op)
\end{align*}
which shows that $\CAlg(\Spec^\op)_{C^\op/}$ is the opposite of $\cC(\Spec)_{/C}$.

In \cite[Section 2.4.3]{lurieHA}, Lurie defines an $\infty$-operad $\mathcal{C}^\amalg$ on an arbitrary $\infty$-category $\mathcal{C}$ and shows this is a symmetric monoidal $\infty$-category structure on $\mathcal{C}$ if and only if $\mathcal{C}$ admits finite coproducts.  In the case of the under category $\CAlg(\Spec^\op)_{C^\op/}$, coproducts are given by pushouts in $\CAlg(\Spec^\op)$,
which in turn are pullbacks in $\cC(\Spec)$ by \cite[Proposition 12.1.7]{RVelements}.  Since $\cC(\Spec)$ is complete, the necessary pullbacks exist.  Thus $\cC(\Spec)_{/C}$ is a Cartesian monoidal $\infty$-category.

Finally, to justify the statement that every object in $\cC(\Spec)_{/C}$ is a cocommutative coalgebra over $C$,  observe that the $\infty$-category 
 \[\cC\big(\cC(\Spec)_{C}\big)\] 
of cocommutative coalgebras in $\cC(\Spec)_{/C}$ is the opposite category of 
\[\CAlg\big((\cC(\Spec)_{C})^\op\big)\] 
of commutative algebras in the opposite category, as in \cite[Definition 3.1.1]{LurieElliptic} and \cref{defncoalspectra,defnHkcoalgspectra}.  Since $(\cC(\Spec)_{/C})^\op$ is coCartesian monoidal, by \cite[Corollary 2.4.3.10]{lurieHA}, every object is a commutative algebra object, and hence every object of $\cC(\Spec)_{/C}$ is a cocommutative coalgebra object.
\end{proof}

We next describe the topological coHochschild homology of a coalgebra spectrum in this framework.   We begin with some basic recollections about Cartesian monoidal $\infty$-categories.

Let $\mathcal{C}$ be a Cartesian monoidal $\infty$-category; in particular this implies $\mathcal{C}$ admits finite products.  For a finite set $X$ and an object $C\in \mathcal{C}$, the $X$-fold product $\prod_X C$ provides a model for the categorical cotensor $C^X$: we have an equivalence of mapping spaces
\[ \Map_\mathcal{C}\big(D, {\textstyle\prod\limits_X C}\big)\simeq (\Map_\mathcal{C}(D,C))^X.\]
Indeed, if $\mathcal{C}$ admits $X$-fold products for an arbitrary set $X$, we may similarly model a cotensor $C^X$ as a product.

This construction is contravariantly functorial in $X$, in the following sense. Suppose $\mathcal{C}$ admits products indexed by any set $X\in \Set$. We may then obtain a map of $\infty$-categories $\Set^\op\xto{C^-} \mathcal{C}$.  As shown in \cite[Proposition 1.1.11]{RVelements}, there is an adjunction
\[\sset(\Set^\op,\mathcal{C})\cong \Cat(\Set^\op,h\mathcal{C})\]
where $h\mathcal{C}$ is the homotopy category of $\mathcal{C}$ and we are, as usual, omitting notation for the nerve of the category $\Set^\op$ on the left. By \cite[Lemma 2.3.3]{RVelements} for a set $X$, $X$-fold products in $\mathcal{C}$ are also products in $h\mathcal{C}$, and so $h\mathcal{C}$ admits $X$-fold products.  The usual argument from the universal properties of products then shows that the assignment $X\mapsto C^X$ is a contravariant functor $\Set^\op\to h\mathcal{C}$; by adjunction we thus have the desired functor of $\infty$-categories.

Given a simplicial set $X_\bullet$ and an object $C$ in an $\infty$-category admitting products, we obtain a $\Delta$-shaped diagram in $C$ via the composite
\[\Delta\xto{X^\op_\bullet}\Set^\op \xto{C^-}\mathcal{C}.\]
This is a cosimplicial object in $\mathcal{C}$ which we denote by $C^{X_\bullet}$ and which we view as the ``cosimplicial cotensor'' of $X_\bullet$ and $C$.   When $\mathcal{C}$ is complete, we may take the limit $\lim_{\Delta} C^{X_\bullet}$ to obtain an object of $\mathcal{C}$. 
  We sometimes use the notation $C^X=\lim_{\Delta} C^{X_\bullet}$, where we remove the $\bullet$ to emphasize that it is a single object of $\mathcal{C}$, rather than a cosimplicial object.  This limit can alternately be described as the totalization of the cosimplicial object $C^{X_\bullet}$.

In the case where $X_\bullet$ is the simplicial set $\Delta^1_\bullet$, this process gives a ``resolution'' of the object $C$.
\begin{lemma}\label{cotensorwithDeltaoneinftycate}
Let $\mathcal{C}$ be a Cartesian monoidal $\infty$-category.  For an object $C\in \mathcal{C}$, let $\Delta\to \mathcal{C}$ be the cosimplicial object of $\mathcal{C}$ given by 
\[ \Delta\xto{(\Delta^1_\bullet)^\op} \Set^\op \xto {C^-}\mathcal{C}.\]
Then $C\simeq \lim_{\Delta} C^{\Delta^1_\bullet}$.
\end{lemma}
\begin{proof}
The simplicial set $\Delta^1_\bullet$ is the representable functor $\Delta(-,[1])$.  This extends to a split augmented simplicial object $\Delta_\bot^\op\to \Set$ with ``extra degeneracies.''  Here $\Delta_\bot$ is the category obtained from $\Delta$ by adding an ``extra degeneracy'' $\sigma^{-1}\colon [n+1]\to [n]$ for each $n$, including a coaugmentation $[0]\to [-1]$.  The inclusion functor $\Delta\to\Delta_\bot$ has a right adjoint \cite[Example B.5.2]{RVelements} and the fact that $\Delta^1_\bullet$ extends to a presheaf on $\Delta_\bot$ follows from the fact that the representing object $[1]$ is in the image of the right adjoint.

 We thus obtain a split coaugmented cosimplicial object
\[ \Delta_\bot \to \Set^\op \xto{C^-} \mathcal{C}\]
that restricts to $C^{\Delta^1_\bullet}$. By \cite[Proposition 2.3.15]{RVelements}, any cosimplicial object that admits a splitting and coaugmentation has a limit given by the coaugmentation---that is, by evaluation at $[-1]\in \Delta_\bot$.  In this case, by construction the coaugmentation is simply the 1-fold product of $C$, meaning the limit is $C$ itself. 
\end{proof}

Since the $\infty$-category $\cC(\Spec)$ is complete and cocomplete, the totalization of any of the cosimplicial objects $C^{X_\bullet}$ exists.  In discussing coalgebras, however, we frequently work at the cosimplicial level rather than passing all the way to the totalization.

\begin{defn}\label{defncosimplcoTHH}
Let $S^1_\bullet$ be the standard model of the simplicial circle as $\Delta^1_\bullet/\partial\Delta^1_\bullet$.  
Let $C$ be in $\cC(\Spec)$.  Then there is a cosimplicial cocommutative coalgebra spectrum $\coTHH^\bullet(C)$ defined by 
\[ \Delta\xto{(S^1_\bullet)^\op} \Set^\op\xto{C^-} \cC(\Spec). \]
\end{defn}

The forgetful functor $\cC(\Spec)\to \Spec$ allows us to view $\coTHH^\bullet(C)$ simply as a cosimplicial spectrum as well. Viewing $\coTHH^\bullet(C)$ as a cosimplicial cotensor will allow us to induce structure on $\coTHH^\bullet(C)$ via maps of the simplicial circle $S^1_{\bullet}$, as we did for $\coHH^\bullet(D)$ in \cref{sec:coHH}. 

\begin{defn}\label{defnspectralcoTHH}
Given a cocommutative coalgebra spectrum $C$, define the spectrum $\coTHH(C)$ to be the (homotopy) limit \emph{in spectra} of the cosimplicial spectrum constructed in \cref{defncosimplcoTHH}:
\[ \coTHH(C)=\lim_\Delta \coTHH^\bullet(C)\]
That is, $\coTHH(C)$ is the totalization in spectra of $\coTHH^\bullet(C)$.
\end{defn}

\begin{rem} In the case where $C$ is a coalgebra spectrum in one of the standard model categories of spectra, \cref{defncosimplcoTHH} reproduces the cosimplicial spectrum $\coTHH^\bullet(C)$ whose totalization defined $\coTHH(C)$ in \cite[Definition 2.2]{BGHSZ} and in \cite{hs.coTHH}.  Hence \cref{defnspectralcoTHH} agrees with the construction of $\coTHH(C)$ in \cite{BGHSZ} and \cite{hs.coTHH}.  

The essential point is that in coalgebra spectra, the comultiplication $C\to C\sma C$ is the universal diagonal map from an object to its two-fold Cartesian product---this is why the symmetric monoidal smash product on coalgebras is the Cartesian product.  Similarly, the empty product is the unit object $S$ for $\sma$ and the map to this terminal object is the counit on $C$.  One then checks that the face and degeneracy maps of $S^1_\bullet$ induce the coface and codegeneracy maps of \cite[Definition 2.2]{BGHSZ}.  See the discussion in \cite[\S 4]{BGHSZ} for more details on this identification.
\end{rem}

\begin{rem}
An alternate way of viewing $\coTHH(C)$ and the cyclic cobar complex $\coTHH^\bullet$ in an $\infty$-categorical framework is supplied by Bay\i nd\i r--P\'eroux \cite{bayindirperoux} who use it to analyze duality relating $\THH$ and $\coTHH$.
\end{rem}

\begin{rem}
In \cref{defnspectralcoTHH}, taking the limit in spectra---as opposed to in cocommutative coalgebra spectra---may seem unnatural, especially to those familiar with the dual case, where McClure, Schw\"{a}nzl, and Vogt \cite{MSV97} prove that for a commutative ring spectrum $A$ there is an equivalence of commutative ring spectra $\THH(A)\simeq A\otimes S^1$. 
In a simplicial model category of coalgebra spectra, the totalization of $\coTHH^\bullet(C)$ would produce the cotensor of $C$ with $S^1_\bullet$.
However, in contrast to the situation for geometric realization of (commutative) algebra spectra, totalization in cocommutative coalgebra spectra does not generally agree with totalization in spectra.  The key difference, of course, is that smash product of spectra commutes with geometric realization but not with totalization.   Hence $\Tot_{\cC(\Spec)}(\coTHH^\bullet(C))\not\simeq \coTHH(C)$.
 
Totalizing in spectra in \cref{defnspectralcoTHH} means that our discussion of $\coTHH$ in this paper agrees with that of \cite{BGHSZ} and extends the work of \cite{hps,hs.coTHH}.  In particular, the coB\"okstedt spectral sequence defined in \cite{BGHSZ}, and further discussed in \cref{sect:thespectralsequence,sect:computations} of this paper, arises from the Bousfield--Kan spectral sequence for totalization in spectra; the results of \cite{BGHSZ} would not apply to a totalization in cocommutative coalgebra spectra.

Furthermore, while the present work focuses on the case of cocommutative coalgebra spectra, the construction of $\coTHH$ generalizes that of $\coHH$, which can be applied to not-necessarily-cocommutative  coalgebras, as is laid out in \cite{BGHSZ}.  Given that totalization in cocommutative coalgebras, coassociative coalgebras, and underlying spectra need not agree, the uniform way to produce $\coTHH$ for all flavors of coalgebra is by totalizing in underlying spectra.
\end{rem}

\section{Coalgebra structure and free loop spaces}\label{sec:freeloop}

The definition of $\coTHH$ as a totalization in spectra, rather than in coalgebra spectra, means that many of the good formal properties enjoyed by $\THH$ do not necessarily carry over to this dual setting.  At heart, this issue arises from the inevitable failure of totalization to commute with smash product.  This means that even for a cocommutative coalgebra spectrum $C$, we cannot easily show that $\coTHH(C)$ is itself a coalgebra spectrum, for example.  However, for the important example of suspension spectra, $\coTHH$ is much better behaved.  In the case of a simply connected space $X$, $\coTHH(\Sigma^\infty_+X)$ is coalgebra spectrum, namely, as shown by Malkiewich \cite{malkiewich} and Hess--Shipley \cite{hs.coTHH},  $\coTHH(\Sigma^\infty_+X)\simeq \Sigma^\infty_+ \freeloops X$.   This identification, on which we elaborate in the following lemma, means that we can obtain significantly more algebraic structure on $\coTHH$ in the suspension spectrum case than in general.

\begin{lemma}\label{lem:comultfreeloops}
Let $X$ be a simply connected space.  Then $\coTHH(\Sigma^\infty_+X)$ has the structure of a coalgebra spectrum with comultiplication induced by the fold map on $S^1_{\bullet}\amalg S^1_{\bullet}$. 
\end{lemma}

\begin{proof}
Let $X$ be a simply connected space.  Consider the cosimplicial space $X^{S^1_\bullet \amalg S^1_\bullet}$.  By inspection,  we have an isomorphism of cosimplicial spaces
\[X^{S^1_\bullet\amalg S^1_\bullet}\cong X^{S^1_\bullet}\times X^{S^1_\bullet}.\]
Since totalization commutes with products, we find
\[\Tot_\Top(X^{S^1_\bullet\amalg S^1_\bullet})\cong \Tot_\Top(X^{S^1_\bullet})\times \Tot_\Top(X^{S^1_\bullet}).\]
The fold map $\nabla\colon S^1_\bullet\amalg S^1_\bullet\to S^1_\bullet$ induces a morphism of simplicial spaces $X^{S^1_\bullet}\to X^{S^1_\bullet \amalg S^1_\bullet}$, which, after totalization, induces the diagonal map 
\[ \Tot_\Top(X^{S^1_\bullet})\to  \Tot_\Top(X^{S^1_\bullet})\times \Tot_\Top(X^{S^1_\bullet}).\]
Since totalization $\Tot_\Top(X^{S^1_\bullet})$ can be identified with the free loop space $\freeloops X$ by \cite{CohenJones}, this map is the diagonal map $\freeloops X\to \freeloops X\times \freeloops X$.

Malkiewich \cite{malkiewich} shows that there is a natural map
\[ \Sigma^\infty_+(\Tot_{\Top}(X^{S^1_\bullet}))\to \Tot_{\Spec}(\Sigma^\infty_+(X^{S^1_\bullet}))\]
where on the right-hand side we have taken the suspension spectrum at each cosimplicial level of the cosimplicial space $X^{S^1_\bullet}$.  Note that at cosimplicial level $n$, we may identify the suspension spectrum $\Sigma^\infty_+(X^{\times n})$ with $(\Sigma^\infty_+ X)^{\sma n}$ and thus this right-hand cosimplicial spectrum is $\coTHH^\bullet(\Sigma_+^\infty X)$.  Malkiewich further shows that this map is an equivalence (after Reedy-fibrantly replacing on the right).  In fact, this map fits into a diagram
\[\xymatrix{
\Sigma^\infty_+ \Tot_{\Top}(X^{S^1_\bullet}) \ar[r]\ar[d]_\nabla & \Tot_{\Spec}(\Sigma^\infty_+(X^{S^1_\bullet}))\ar[d]^\nabla\\
\Sigma^\infty_+\Tot_{\Top}(X^{S^1_\bullet \amalg S^1_\bullet})\ar[r] & \Tot_{\Spec}(\Sigma^\infty_+(X^{S^1_\bullet \amalg S^1_\bullet})).
}
\] 
The lower horizontal map is also an equivalence, as we can see via the identification of $X^{S^1_\bullet\amalg S^1_\bullet}$ with $(X\times X)^{S^1_\bullet}$, to which Malkiewich's result applies.
Using the above identifications, we further obtain the equivalence
\begin{align*}
\Tot_{\Spec}(\Sigma^\infty_+(X^{S^1_\bullet\amalg S^1_\bullet}))&\simeq \Sigma^\infty_+\Tot_{\Top}(X^{S^1_\bullet \amalg S^1_\bullet})\\
&\simeq\Sigma^\infty_+(\Tot_{\Top}(X^{S^1_\bullet})\times \Tot_{\Top}(X^{S^1_\bullet}))\\
&\simeq\Sigma^\infty_+\Tot_{\Top}(X^{S^1_\bullet})\sma \Sigma^\infty_+\Tot_{\Top}(X^{S^1_\bullet})\\
&\simeq\Tot_{\Spec}(\Sigma^\infty_+ (X^{S^1_\bullet}))\sma \Tot_{\Spec}(\Sigma^\infty_+(X^{S^1_\bullet})).
\end{align*}
Thus the right vertical map in the diagram is a comultiplication of spectra 
\[ \triangle\colon \coTHH(\Sigma^\infty_+ X)\to \coTHH(\Sigma^\infty_+ X)\sma \coTHH(\Sigma^\infty_+ X)\]
 and in fact agrees with the canonical comultiplication we obtain on $\coTHH(\Sigma^\infty_+X)$  from the diagonal map after using the identification $\coTHH(\Sigma^\infty_+X)\simeq \Sigma^\infty_+\freeloops X$.

We additionally have a counit map $\coTHH(\Sigma^\infty_+X)\to \coTHH(\Sigma^\infty_+\ast) $ that arises from the collapse map $\pi\colon X_+\to \ast_+=S^0$; note that $\coTHH(\Sigma^\infty_+\ast)=\coTHH(S)\simeq S$ either from Malkiewich's result about free loop spaces or from the observation that $\coTHH^\bullet(S)$ is the constant cosimplicial spectrum at $S$.    The axioms for a coalgebra spectrum thus can be verified at the cosimplicial level. 

For example, we show counitality explicitly.  At the cosimplicial level, the composite
\[ X^{S^1_\bullet} \xto{\nabla} X^{S^1_\bullet \amalg S^1_\bullet} \simeq X^{S^1_\bullet}\times X^{S^1_\bullet}\xto{\pi\times \id} \ast^{S^1_\bullet}\times X^{S^1_\bullet}\simeq X^{S^1_\bullet}
\]
is the identity map. 
After passing to suspension spectra and totalizing, we obtain the left counitality condition.
\end{proof}
We now make two important observations following from the identifications above. 
\begin{rem}\label{usefulremarkaboutcothhsuspspec1}
In the case of the cosimplicial spectrum $(\Sigma^\infty_+X)^{S^1_\bullet}$, totalization commutes with the smash product 
\[ \Tot((\Sigma^\infty_+X)^{S^1_\bullet}\sma (\Sigma^\infty_+X)^{S^1_\bullet})\simeq \Tot((\Sigma^\infty_+X)^{S^1_\bullet})\sma \Tot((\Sigma^\infty_+X)^{S^1_\bullet}).\]
By the above, both sides are weakly equivalent to $\Sigma^\infty_+\freeloops X\sma \Sigma^\infty_+\freeloops X$.
\end{rem}
\begin{rem}\label{usefulremarkaboutcothhsuspspec2}
The totalization of $(\Sigma^\infty_+ X)^{S^1_\bullet}$ as a coalgebra spectrum (in the infinity categorical framework) agrees with the totalization as a spectrum.  For this, we observe that by the above lemma, the totalization in spectra is already a coalgebra spectrum, and its coalgebra structure maps come from the structure maps on the cosimplicial coalgebra spectrum $\Sigma^\infty_+ X^{S^1_\bullet}$.  Hence the totalization in spectra $\Tot(\Sigma^\infty_+ X^{S^1_\bullet})$ already satisfies the universal property of the totalization in coalgebra spectra and the natural map
\[ \Tot_{\cC}((\Sigma^\infty_+X)^{S^1_\bullet})\to \Tot_{\Spec}((\Sigma^\infty_+X)^{S^1_\bullet})\]
must be an equivalence.
\end{rem}

The result of \cref{lem:comultfreeloops} is analogous to the result for topological Hochschild homology that when $R$ is commutative, $\THH(R)$ is an algebra spectrum. Further, for commutative $R$, $\THH(R)$ is known to be a Hopf algebra over $R$ in the homotopy category \cite{ekmm, MSV97}. Below, we consider the analogous result for the topological coHochschild homology of suspension spectra.  Our proof relies on the identification of the coalgebra structure on $\coTHH(\Sigma^\infty_+ X)$ from \cref{lem:comultfreeloops} and in particular on the consequence of this structure noted in \cref{usefulremarkaboutcothhsuspspec2}.  This allows us to work in the context of cocommutative coalgebra spectra and we take advantage of additional formal properties available there.

\begin{prop}\label{cothhiscoalgoversuspspec}
For a simply connected space $X$, $\coTHH(\Sigma^\infty_+ X)$ is a $\Box_{\Sigma^\infty_+X}$-coalgebra; that is, $\coTHH(\Sigma^\infty_+ X)$ defines an object in the $\infty$-category \[\cC\big(\cC(\Spec)_{/\Sigma^\infty_+X}\big).\]
\end{prop}

\begin{proof}
Observe that the inclusion of the basepoint $\ast\to S^1_\bullet$ induces a map of cosimplicial coalgebra spectra
\[\coTHH^\bullet(\Sigma^\infty_+ X)\to \Sigma^\infty_+ X\]
where the target is the constant cosimplicial coalgebra spectrum.  Totalizing in  $\cC(\Spec)$ we obtain a map $\coTHH(\Sigma^\infty_+X)\to \Sigma^\infty_+ X$ and by \cref{lem:comultfreeloops}, this is a map of cocommutative coalgebra spectra.  Hence $\coTHH(\Sigma^\infty_+ X)$ is an object of $\cC(\Spec)_{/\Sigma^\infty_+X}$.  \cref{inftyboxprodstructurecoalg} shows that $\coTHH(\Sigma^\infty_+X)$ is then a cocommutative $\Box_{\Sigma^\infty_+X}$-coalgebra.
 \end{proof}

It is useful to be more explicit in identifying the comultiplication structure on $\coTHH(\Sigma^\infty_+ X)$. Unwinding the proof of \cref{inftyboxprodstructurecoalg}, the comultiplication over $\Sigma^\infty_+X$  on $\coTHH(\Sigma^\infty_+ X)$ of \cref{cothhiscoalgoversuspspec} is a  map of cocommutative coalgebra spectra
\[ \coTHH(\Sigma^\infty_+ X)\to \coTHH(\Sigma^\infty_+X)\Box_{\Sigma^\infty_+X} \coTHH(\Sigma^\infty_+X)\]
where the target is the monoidal product in $\cC(\Spec)_{/\Sigma^\infty_+ X}$, defined by pullback over $\Sigma^\infty_+X$.  Since $S^1_\bullet \vee S^1_\bullet$ is the pushout
\[\xymatrix{ \ast \ar[r]\ar[d]& S^1_\bullet\ar[d]\\
S^1_\bullet \ar[r]& S^1_\bullet \vee S^1_\bullet,}
\]
$(\Sigma^\infty_+ X)^{S^1_\bullet \vee S^1_\bullet}$ is the pullback in cosimplicial coalgebra spectra
\[ \xymatrix{ (\Sigma^\infty_+ X)^{S^1_\bullet\vee S^1_\bullet}\ar[r]\ar[d] & (\Sigma^\infty_+ X)^{S^1_\bullet}\ar[d]\\
(\Sigma^\infty_+ X)^{S^1_\bullet} \ar[r]&(\Sigma^\infty_+ X). }
\]
Totalization preserves pullbacks in $\infty$-categories, so we have a pullback diagram in cocommutative coalgebra spectra
\[ \xymatrix{ \Tot_\cC((\Sigma^\infty_+ X)^{S^1_\bullet\vee S^1_\bullet})\ar[r]\ar[d] & \Tot_\cC((\Sigma^\infty_+ X)^{S^1_\bullet})\ar[d]\\
\Tot_\cC((\Sigma^\infty_+ X)^{S^1_\bullet}) \ar[r]&\Tot_\cC(\Sigma^\infty_+ X). }
\]
The totalization in spectra of $(\Sigma^\infty_+ X)^{S^1_\bullet}$ and $(\Sigma^\infty_+ X)^\ast$ are already cocommutative coalgebras, so we identify this pullback as
\begin{equation}\label{htpypbdefiningcoTHHboxcoTHH}
\begin{gathered} \xymatrix{ \Tot_\cC((\Sigma^\infty_+ X)^{S^1_\bullet\vee S^1_\bullet})\ar[r]\ar[d] & \coTHH(\Sigma^\infty_+ X)\ar[d]\\
\coTHH(\Sigma^\infty_+ X) \ar[r]&\Sigma^\infty_+ X,}
\end{gathered}
\end{equation}
and thus we have made the identification
\[ \coTHH(\Sigma^\infty_+ X)\Box_{\Sigma^\infty_+X}\coTHH(\Sigma^\infty_+ X) =\Tot_{\cC{\Spec}}(\Sigma^\infty_+X)^{S^1_\bullet \vee S^1_\bullet}.\]
 The comultiplication map $\coTHH(\Sigma^\infty_+X)\to \coTHH(\Sigma^\infty_+ X)\Box_{\Sigma^\infty_+ X}\coTHH(\Sigma^\infty_+ X)$ is therefore the totalization in $\cC(\Spec)$ of the map of cosimplicial coalgebra spectra
\[ (\Sigma^\infty_+ X)^{S^1_\bullet} \to (\Sigma^\infty_+ X)^{S^1_\bullet \vee S^1_\bullet}\]
obtained by cotensoring with the fold map $S^1_\bullet \vee S^1_\bullet\to S^1_\bullet$ of simplicial sets.

The cocommutative coalgebra structure on $\coTHH(\Sigma^\infty_+X)$ is thus fairly formal in the sense that it does not use any special properties of $S^1_\bullet$. Showing that $\coTHH(\Sigma^\infty_+ X)$ admits multiplication-like structure does require specific properties of the circle.  In particular, as in \cref{sec:coHH}, this requires using alternate simplicial models of the circle.  We first show that $\coTHH(\Sigma^\infty_+ X)$ can also be constructed using a double circle model of $S^1$.

\begin{lemma}\label{doublecirclearefine} Let $\mathcal{C}$ be a Cartesian monoidal $\infty$-category. Recall the double circle  $dS^1_\bullet =(\Delta^1\amalg \Delta^1)\amalg_{(\partial\Delta^1\amalg\partial\Delta^1)}\partial\Delta^1$ defined in \cref{sec:coHH}, and let $\pi\colon dS^1_\bullet \to S^1_\bullet$ be the map of simplicial sets given by collapsing one of the 1-simplices.  For an object $C\in \mathcal{C}$, totalization in $\mathcal{C}$ induces an equivalence
\[ \lim_\Delta C^{S^1_\bullet} \to \lim_\Delta C^{dS^1_\bullet}.\]
\end{lemma}
\begin{proof}
The collapse map $dS^1_\bullet\to S^1_\bullet$ is the induced map on pushouts of the diagrams
\[\xymatrix{
\Delta^1_\bullet \ar@{=}[d] & \partial \Delta^1_\bullet \ar@{=}[d]\ar[r]\ar[l]& \Delta^1_\bullet \ar[d]\\
\Delta^1_\bullet &\partial \Delta^1_\bullet \ar[l] \ar[r]&\ast
}
\]
Hence the map $C^{S^1_\bullet}\to C^{dS^1_\bullet}$ of cosimplicial objects of $\mathcal{C}$ is the induced map on the pullbacks below
\[\xymatrix{
C^{\Delta^1_\bullet} \ar[r] & C^{\partial \Delta^1_\bullet} & C^{\Delta^1_\bullet} \ar[l]\\
C^{\Delta^1_\bullet}\ar@{=}[u] \ar[r]&C^{\partial \Delta^1_\bullet}   \ar@{=}[u] & C^\ast\ar[u]\ar[l]
}
\]
We are interested in the map $\lim_\Delta C^{S^1_\bullet} \to \lim_\Delta C^{dS^1_\bullet}.$ Since limits commute, as follows from \cite[Lemma 2.4.1]{RVelements}, it is equivalent to calculate the pullback of the diagram of totalizations in $\mathcal{C}$:
\[\xymatrix{
\lim_\Delta C^{\Delta^1_\bullet} \ar[r] & \lim_\Delta C^{\partial \Delta^1_\bullet} & \lim_\Delta C^{\Delta^1_\bullet} \ar[l]\\
\lim_\Delta C^{\Delta^1_\bullet}\ar@{=}[u] \ar[r]&\lim_\Delta C^{\partial \Delta^1_\bullet}   \ar@{=}[u] & \lim_\Delta C^\ast\ar[u]^{\sim}\ar[l]
}
\]
By \cref{cotensorwithDeltaoneinftycate}, the right vertical map is an equivalence, and thus the pullbacks of these equivalent diagrams are equivalent.
\end{proof}

\begin{prop}\label{prop:suspspectrum_boxcoalgspectrallevel}
Let $X$ be a simply connected space.  Then $\coTHH(\Sigma^\infty_+ X)$ is a Hopf monoid in the homotopy category of coalgebra spectra over $\Sigma^\infty_+X$: it is a bimonoid with an antipode map.  
\end{prop}
Intuitively, this structure should be thought of as that of a $\Box_{\Sigma^\infty_+ X}$-Hopf algebra structure on $\coTHH(\Sigma^\infty_+ X)$ up to homotopy.

\begin{proof}
We must show that $\coTHH(\Sigma^\infty_+ X)$ has the structure of a comonoid and a monoid and define the antipode map relating these structures. The comonoid structure is that of \cref{cothhiscoalgoversuspspec}, with comultiplication
\[ \coTHH(\Sigma^\infty_+X)\to \coTHH(\Sigma^\infty_+X)\Box_{\Sigma^\infty_+ X} \coTHH(\Sigma^\infty_+ X)\]
identified as arising from the fold map $S^1_\bullet \vee S^1_\bullet \to S^1_\bullet$ of the simplicial circle and counit identified as arising from the inclusion of the basepoint $\ast\to S^1_\bullet$.  

We next define the monoid structure on $\coTHH(\Sigma^\infty_+X)$, over $\Sigma^\infty_+ X$.  Recall that there is a based simplicial pinch map
\[
\psi\colon dS^1_\bullet \to S^1_\bullet \vee S^1_\bullet, 
\]
where $dS^1_\bullet$ is the double circle as above.  
This induces maps of cosimplicial coalgebra spectra
\[\xymatrix{
(\Sigma^\infty_+ X)^{S^1_\bullet \vee S^1_\bullet} \ar[r]^{\psi}\ar[dr]& (\Sigma^\infty_+ X)^{dS^1_\bullet}\ar[d]\\
& (\Sigma^\infty_+X)^*
}
\]
After totalization in $\cC(\Spec)$ this yields a map over $\Sigma^\infty_+X$
\[
\psi\colon \coTHH(\Sigma^\infty_+ X)\Box_{\Sigma^\infty_+ X} \coTHH(\Sigma^\infty_+ X) \to \Tot_{\cC}((\Sigma^\infty_+ X)^{dS^1_\bullet}).
\]
The collapse map $\pi\colon dS^1_\bullet \to S^1_\bullet$ that takes the second $\Delta^1_\bullet$ to the basepoint induces a map
\[
\pi: \Tot_{\cC}(\Sigma^\infty_+ X)^{S^1_\bullet} \to \Tot_{\cC}(\Sigma^\infty_+ X)^{dS^1_\bullet}
\]
which is an equivalence in $\cC(\Spec)$ by \cref{doublecirclearefine} and thus in spectra.  
Thus, by \cref{usefulremarkaboutcothhsuspspec2}, the target of $\psi$ is weakly equivalent to $\coTHH(\Sigma^\infty_+ X)$, so we may view $\psi$ as the multiplication of a monoid structure.  The unit map $\Sigma^\infty_+ X\to \coTHH(\Sigma^\infty_+ X)$ is obtained similarly from the simplicial map $S^1_\bullet \to \ast$.  The monoid axioms arise from suitably commuting diagrams or homotopy-commuting diagrams of simplicial sets, as in the proof of \cref{coHHbialg2}. 

The antipode map $\chi\colon\coTHH(\Sigma^\infty_+X)\to \coTHH(\Sigma^\infty_+ X)$ is  defined using the double circle $dS^1_\bullet$.  The double circle enjoys  a simplicial flip map $\chi'\colon dS^1_\bullet \to dS^1_\bullet$.  Cotensoring with $\chi'$ and using the equivalence $\pi$ produces the zigzag 
\[\xymatrix{
& *=<7em,3ex>{ \Tot_{\cC}((\Sigma_+^\infty X)^{dS^1_\bullet})}\ar[]!<5.1em,0ex>;[rr]!<-5.1em,0ex>^{\chi'}& &*=<7em,3ex>{\Tot_{\cC}((\Sigma_+^\infty X)^{dS^1_\bullet})} \\
*=<0em,3ex>{\Tot_{\cC}( (\Sigma_+^\infty X)^{S^1_\bullet})}\ar[]!<0em,2ex>;[ur]!<-2em,-2ex>^-{\pi}&&&&*=<0em,3em>{\Tot_{\cC}(( \Sigma_+^\infty X)^{S^1_\bullet})}\ar[]!<-0em,2ex>;[ul]!<2em,-2ex>_-{\pi}
}
\]
and hence an antipode map $\chi\colon \coTHH(\Sigma^\infty_+X)\to \coTHH(\Sigma^\infty_+ X)$ in the homotopy category. 

To check the compatibility diagram involving the antipode, we also need a triple model for the circle. 
As in \cite{AR}, let $tS^1_{\bullet}$ denote $\partial\Delta^2$, with three non-degenerate 1-simplices as shown:\\
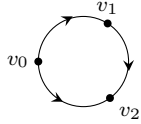
\centerline{\begin{tikzpicture}[>=stealth]
    \draw[ 
        decoration={markings, mark=at position 0.005 with {\arrow[thick]{<}},mark=at position 0.16 with {\fill circle (1.5pt) node[above]{\footnotesize $v_1$};}, mark=at position 0.325 with {\arrow[thick]{<}},mark=at position 0.5 with {\fill circle (1.5pt) node[left] {\footnotesize$v_0$};}, mark=at position 0.675 with {\arrow[thick]{>}}, mark=at position 0.85 with {\fill circle (1.5pt) node[below right] {\footnotesize $v_2$};},},
        postaction={decorate}
        ]
        (3,0) circle (0.6);
\end{tikzpicture}}
Let  
\[
\psi_1\colon tS^1_{\bullet} \to S^1_\bullet \vee dS^1_\bullet 
\] 
be the simplicial map that identifies the $v_0$ and $v_1$ in $\partial\Delta^2$ and takes the face opposite $v_0$ to the first $\Delta^1$ in $dS^1$. Let 
\[
\psi_2\colon tS^1_{\bullet} \to  dS^1_\bullet \vee S^1_\bullet
\] 
be the simplicial map that identifies the vertices $v_1$ and $v_2$ in  $\partial\Delta^2$ and takes the face opposite $v_2$ to the first $\Delta^1$ in $dS^1.$ Then the commutative diagram in \cref{defn:boxhopf} comes from cotensoring with the following homotopy commutative diagram of simplicial sets:
\[\xymatrix{&  S^1_\bullet \vee dS^1_\bullet \ar[dl]_{\nabla \circ (\id\vee \pi)}&& S^1_\bullet \vee dS^1_\bullet \ar[ll]_{\id\vee \chi'}\\
S^1_\bullet  && \ast_\bullet \ar[ll]_{\epsilon}&& tS^1_\bullet \ar[ll]_{\eta}\ar[ul]_{\psi_1}\ar[dl]^{\psi_2}\\
&dS^1_\bullet \vee S^1_\bullet \ar[ul]^{\nabla\circ (\pi\vee\id)}&& dS^1_\bullet \vee S^1_\bullet \ar[ll]_{\chi'\vee\id}
}
\]
This diagram appears in the proof of \cite[Theorem 3.9]{AR}, where the authors discuss the homotopy commutativity.  Note that the names of the maps $\epsilon$ and $\eta$ are reversed from their choice of names because these maps play opposite roles in the coalgebra and algebra cases.  An argument analogous to that of \cref{doublecirclearefine} shows that cotensoring $\Sigma^\infty_+X$ with $tS^1_\bullet$ produces a coalgebra spectrum that is again equivalent to $\coTHH(\Sigma^\infty_+X)$.
\end{proof}

Before looking at the implications of this result for the homology of free loop spaces, we prove the following lemma.  Recall that a map of coalgebras $A\to C$ allows us to view $A$ as a right or left $C$-comodule, similar to \cref{boxprodforcoalg}.

\begin{lemma}\label{prop-coflat}
Let $k$ be a field and let $A$, $B$ and $C$ be $Hk$-coalgebras. Let $X$ be the homotopy pullback in $Hk$-module spectra of the diagram 
\[
\xymatrix{ & B \ar[d] \\
A \ar[r] & C
}
\]
 If $\pi_*(A)$ or $\pi_*(B)$ is coflat as a comodule over $\pi_*(C)$,  then
\[\pi_*(X)\cong \pi_*(A)\square_{\pi_*(C)} \pi_*(B)\] 
\end{lemma}

\begin{proof}
Consider the cosimplicial cobar spectrum $\Omega^{\bullet}_{Hk}(A, C, B)$ with $n$th cosimplicial level 
\[A \sm_{Hk}  C^{\sm_{Hk} (n)} \sm_{Hk} B,\]   
as discussed, for example, in \cite[\S 3.3]{hs.coTHH}.  Applying homotopy to this cobar spectrum produces a graded cobar complex; that is, 
\[\pi_*(\Omega^{\bullet}_{Hk}(A,  C,  B)) \iso \Omega^{\bullet}_{k}(\pi_*(A), \pi_*(C), \pi_*(B))\] with $n$th cosimplicial level
\[\pi_*(A)\otimes_k \pi_*(C)^{\otimes_k n} \otimes_k \pi_*(B).\]

Consider the Bousfield--Kan spectral sequence associated to the cosimplicial spectrum $\Omega^{\bullet}_{Hk}( A, C,  B).$   To calculate the $E_2$-page, note that if we replace $A$ by $C$, then $\Omega^{\bullet}_{k}(\pi_*(C), \pi_*(C), \pi_*(B))$ produces an injective resolution of $\pi_*(B)$  and applying the functor  
\[\pi_*(A) \square_{\pi_*(C)} -\] recovers $\Omega^{\bullet}_{k}(\pi_*(A), \pi_*(C), \pi_*(B)).$   Hence the $E_2$ page is given by 
\[\cotor_{\pi_*(C)}(\pi_*(A),  \pi_*(B)).\]
 If $\pi_*(A)$ or $\pi_*(B)$ is coflat  over  $\pi_*(C)$, then the spectral sequence collapses with \[\pi_*(A)\square_{\pi_*(C)} \pi_*(B)\]  in the first column and zero everywhere else.  This is the right-hand side above, so the proposition follows by showing that the spectral sequence converges to the left hand side. Since the spectral sequence collapses under the coflatness conditions, it converges completely to $\pi_* \Tot \Omega^{\bullet}_{Hk}(A, C, B) \cong \pi_*(X)$ by \cite[IX.5.4]{BK}. 
\end{proof}

\cref{prop:suspspectrum_boxcoalgspectrallevel} then yields the following algebraic structure on the homology of free loop spaces. 
\begin{cor}
For $X$ a simply connected space and $k$ a field, if $H_*(\freeloops X;k)$ is coflat as a comodule over $H_*(X;k)$, then $H_*(\freeloops X;k)$ is a $\Box_{H_*(X;k)}$-Hopf algebra. 
\end{cor}

\begin{proof} Recall that for a simply connected space $X$, $\coTHH(\Sigma_+^{\infty}X) \simeq \Sigma_+^{\infty}\freeloops X$. The corollary follows from \cref{prop:suspspectrum_boxcoalgspectrallevel} by applying homology once we identify 
\[H_*\big(\coTHH(\Sigma^\infty_+X)\Box_{\Sigma^\infty_+ X}\coTHH(\Sigma^\infty_+X);k\big) \]
with 
\[
H_*(\coTHH(\Sigma^\infty_*X);k)\Box_{H_*(X;k)}H_*(\coTHH(\Sigma^\infty_+X);k).\]
This identification follows from the stability of the $\infty$-category of coalgebra spectra.  Stability implies that the homotopy pullback square in Diagram (\ref{htpypbdefiningcoTHHboxcoTHH}) is also a homotopy pushout square.  After smashing with $Hk$, we still have a homotopy pushout and, by stability, also a homotopy pullback.  Passing to homotopy groups and applying \cref{prop-coflat} above then gives the identification.

We thus obtain a comultiplication map
\[ H_*(\coTHH(\Sigma^\infty_+X);k)\to H_*(\coTHH(\Sigma^\infty_+X);k)\Box_{H_*(X;k)}H_*(\coTHH(\Sigma^\infty_+ X);k)\] by applying homology to the comultiplication map produced in \cref{prop:suspspectrum_boxcoalgspectrallevel}, and likewise for the remaining structure maps of a $\Box_{H_*(X;k)}$-Hopf algebra.
\end{proof}

\begin{rem}
One can show that the Hopf structure of \cref{prop:suspspectrum_boxcoalgspectrallevel} arises from suspending space level structure, in the following sense. As mentioned at the beginning of this section, Hess and Shipley \cite[Appendix A]{hs.coTHH} prove that for $X$ simply connected, $\coTHH(\Sigma^\infty_+ X)\simeq \Sigma^\infty_+(\freeloops X)$.  The main tool is a convergence result due to Bousfield and the argument of \cite[Appendix A]{hs.coTHH} applies directly to show that for $X$ simply connected,
\[ \Tot_\Spec( \Sigma^\infty_+(\Map(S^1_\bullet\vee S^1_\bullet, X)))\simeq \Sigma^\infty_+(\Tot_{\Top}\Map(S^1_\bullet\vee S^1_\bullet,X)).\]
There is a standard identification of the right hand side as $\Sigma^\infty_+ (\freeloops X\times_X\freeloops X)$, and the fiberwise diagonal and concatenation product of string topology give a Hopf-monoid structure for $\freeloops X$ as a space over $X$.  Since this structure arises from the same types of fold and diagonal maps on $S^1\vee S^1$ as  appear in the structure from \cref{prop:suspspectrum_boxcoalgspectrallevel}, we may identify the two structures provided that
\[\Tot_\Spec( \Sigma^\infty_+(\Map(S^1_\bullet\vee S^1_\bullet, X)))\simeq \Sigma^\infty_+\freeloops X\Box_{\Sigma^\infty_+ X} \Sigma^\infty_+ \freeloops X.\]
Because $\Sigma^\infty_+$ is a product-preserving functor from unbased spaces to cocommutative coalgebra spectra, we observe that $\Sigma^\infty_+(\Map(S^1_\bullet\vee S^1_\bullet, X))$ is the underlying cosimplicial spectrum of the cosimplical coalgebra spectrum $(\Sigma^\infty_+ X)^{S^1_\bullet\vee S^1_\bullet}$.  The Hess--Shipley argument shows that the totalization in spectra is a suspension spectrum, and hence already a coalgebra spectrum, and therefore
\[\Tot_\Spec(\Sigma^\infty_+(\Map(S^1_\bullet\vee S^1_\bullet, X)))\simeq \Tot_\cC( (\Sigma^\infty_+ X)^{S^1_\bullet\vee S^1_\bullet}).\]
Then the identification of Diagram (\ref{htpypbdefiningcoTHHboxcoTHH}) shows that the right hand side above is precisely $\Sigma^\infty_+\freeloops X\Box_{\Sigma^\infty_+X}\Sigma^\infty_+ \freeloops X$.  This gives a more ``geometric'' description of this Hopf structure, but we require suitable cosimplicial descriptions in order to use this structure in the spectral sequence results of the next section.
\end{rem}

\section{Hopf Structure in the coB\"okstedt Spectral Sequence}\label{sect:thespectralsequence}

An essential tool for computing topological Hochschild homology is the B\"okstedt spectral sequence of \cite{Bo2}. In \cite{BGHSZ}, we showed that there is an analogous coB\"okstedt spectral sequence for computing topological coHochschild homology.  In this section, we show that this spectral sequence has additional algebraic structure, which we exploit in \cref{sect:computations} to make free loop space computations.  We first recall the structure and convergence results about the coB\"okstedt spectral sequence from \cite{BGHSZ}.

\begin{thm}[\cite{BGHSZ}]\label{thm:SpecSeq} Let $k$ be a field. Let $C$ be a coalgebra spectrum.  The Bousfield--Kan spectral sequence for the cosimplicial spectrum $\coTHH^\bullet(C)$ gives a \emph{coB\"okstedt spectral sequence}  for calculating $H_{t-s}(\coTHH(C);k)$ with  $E_2$-page 
\[ E_2^{s, t}=\coHH^{k}_{s,t}(H_*(C;k))\]
given by the classical coHochschild homology of $H_*(C;k)$.
\end{thm}

We also have the following convergence results.  

\begin{prop}[\cite{BGHSZ}]\label{prop:bkconvergence}
If for each $s$ there is an $r$ such that $E_r^{s, s+i} = E_{\infty}^{s, s+i}$, then the coB\"okstedt spectral sequence for $\coTHH(C)$ converges completely to \[\pi_* \Tot (\coTHH^\bullet(C) \wedge Hk).  \]
\end{prop} 

\begin{rem}
In \cite{BGHSZ}, the authors use the simplicial model category structure on spectra to define the homotopy limits/totalizations and homotopy types used in defining $\coTHH(C)$.  In this paper, we choose instead to use the $\infty$-categorical framework for these constructions, as is reflected in the omission of the (co)fibrancy conditions that appeared in the statements of these results in \cite{BGHSZ}.
\end{rem}

Note that the general construction of a map of the form $\Hom(X,Y)\sma Z \to \Hom(X,Y\sma Z)$ yields a natural map
\[P\colon \Tot(\coTHH^\bullet(C)) \wedge Hk \to \Tot (\coTHH^\bullet(C) \wedge Hk).\]
From \cite{BGHSZ}, if this map $P$ is an isomorphism in homotopy, and the conditions on $E_r^{s, s+i}$ in \cref{prop:bkconvergence} hold, then the coB\"okstedt spectral sequence for $\coTHH(C)$ converges completely to $H_{*}(\coTHH(C);k)$.  

Further, in \cite{BGHSZ} it is shown that the coB\"okstedt spectral sequence of \cref{thm:SpecSeq} is a spectral sequence of coalgebras. Hence if $C$ is a connected cocommutative coalgebra, for each $r>1$ there is a comultiplication
\[
\psi\colon E_r^{**} \rightarrow E_r^{**} \otimes_{k} E_r^{**},
\]
and the differentials $d_r$ respect the comultiplication. 

Based on algebraic structures in the classical B\"okstedt spectral sequence, one might expect to have additional algebraic structure in the coB\"okstedt spectral sequence. In particular, in \cite{AR} Angeltveit and Rognes show that, under a flatness assumption,  the B\"okstedt spectral sequence
\[
E^2_{*,*} = \HH_*(H_*(R; \mathbb{F}_p)) \Rightarrow H_*(\THH(R); \mathbb{F}_p)
\]
is a spectral sequence of Hopf algebras over $H_*(R;\mathbb{F}_p)$. In this section we consider what the analogous algebraic structure is on the coB\"okstedt spectral sequence. Under coflatness conditions, we prove that the coB\"okstedt spectral sequence is a spectral sequence of $\boxhopf{H_*(C;k)}$s, in the sense of the following definitions. 

\begin{defn}
Let $D$ be a cocommutative coalgebra over a field $k$. A \emph{differential bigraded $\Box_D$-algebra $(E^{*,*}, d)$} is a bigraded $D$-bicomodule $E^{*,*}$, a map of $D$-bicomodules 
\[
d\colon \bigoplus_{q-p=n} E^{p,q} \to \bigoplus_{s-r=n-1} E^{r, s},
\] and a $\Box_D$-multiplication structure
\[
\mu\colon E^{s,t} \Box_D E^{u,v} \to E^{s+u, t+v} 
\]
with a unit $\eta\colon D \to E^{*,*}$, such that the usual associativity and unitality diagrams commute.  The differential is compatible with the product, in the sense that $d$ must satisfy the Leibniz rule:
\[
d \circ \mu = \mu \circ (d \Box_D \id + (-1)^{s+t} \id \Box_D d).
\]
\end{defn}

\begin{defn}
Let $D$ be a cocommutative coalgebra over a field $k$. We say that a spectral sequence $\{E_r, d_r\}$ is a \emph{spectral sequence of $\Box_D$-algebras} if every $D$-comodule $E_r^{s,t}$ is coflat, and for every $r\geq 1$, $(E_r^{*,*}, d_r)$ is a differential bigraded $\Box_D$-algebra, with multiplication $\mu_r$, and if the multiplication $\mu_{r+1}$ is the composite
\begin{multline*}
\mu_{r+1}\colon E_{r+1} \Box_D E_{r+1} \cong H_*(E_r; k) \Box_D H_*(E_r; k) \\\cong H_*(E_r \Box_D E_r; k) \xto{H_*(\mu_r)}  H_*(E_r;k) \cong E_{r+1}.
\end{multline*}
Here the isomorphism $H_*(E_r;k) \Box_D H_*(E_r;k) \cong H_*(E_r \Box_D E_r;k)$ is the K\"unneth isomorphism for $\Box_D$, as in \cref{thrm:boxkunneth}. This uses the hypothesis that $E_r^{s,t}$ is coflat.
\end{defn}

We now also define the notion of a spectral sequence of $\Box_D$-coalgebras. 
\begin{defn}\label{defn:diffbigradedboxcoalg}
Let $D$ be a cocommutative coalgebra over a field $k$. A \emph{differential bigraded $\Box_D$-coalgebra $(E^{*,*}, d)$} is a bigraded $D$-bicomodule $E^{*,*}$, a map of $D$-bicomodules
\[
d\colon \bigoplus_{q-p=n} E^{p,q} \to \bigoplus_{s-r=n-1} E^{r, s},
\] and  a $\Box_D$-comultiplication structure
\[
\triangle\colon E^{s,t}  \to \bigoplus_{\substack{u+w=s\\v+x=t}}E^{u, v} \Box_D E^{w,x} 
\]
with a counit $\epsilon\colon E^{*,*}\to D$, such that the usual coassociativity and counitality diagrams commute. The differential is compatible with the comultiplication, in the sense that $d$ must satisfy the coLeibniz rule:
\[
\triangle \circ d = \big(d\Box \id + (-1)^{u+v}\id \Box d\big) \circ \triangle.
\]
\end{defn}

\begin{defn}\label{defn:boxcoalgss}
Let $D$ be a cocommutative coalgebra over a field $k$. We say that a spectral sequence $\{E_r, d_r\}$ is a \emph{spectral sequence of $\Box_D$-coalgebras} if for every $r\geq 1$, $(E_r^{*,*}, d_r)$ is a differential bigraded $\Box_D$-coalgebra, with comultiplication $\triangle_r$, and if the comultiplication $\triangle_{r+1}$ is the composite
\begin{multline*}
\triangle_{r+1}\colon E_{r+1} \cong H_*(E_r;k) \xto{H_*(\triangle_r)}  H_*(E_r \Box_D E_r;k) \\\xto{\ \phi\ }  H_*(E_r;k) \Box_D H_*(E_r;k) \cong  E_{r+1} \Box_D E_{r+1}.
\end{multline*}
Here the map $\phi$ is the map from the homology of the cotensor to the cotensor of the homologies, as in \cref{thrm:boxkunneth}. 
\end{defn}

We now characterize a $\Box$-Hopf algebra structure on a spectral sequence. 
\begin{defn}\label{defn:diffbigradedboxhopf}
Let $D$ be a cocommutative coalgebra over a field $k$. A \emph{differential bigraded $\Box_D$-Hopf algebra} is a pair $(E^{*,*}, d)$ with the structure of both a differential bigraded $\Box_D$-algebra and a differential bigraded $\Box_D$-coalgebra. 
The multiplication, unit, comultiplication and counit must be compatible, as in \cref{def:boxbialg}. Further, a differential bigraded $\Box_D$-Hopf algebra must have differential $D$-bicomodule maps 
\[
\chi\colon E^{s,t} \to E^{s,t}
\]
satisfying the commutative diagram of \cref{defn:boxhopf}.  \end{defn}

\begin{defn}\label{defn:boxHopfss}
Let $D$ be a cocommutative coalgebra over a field $k$. We say that a spectral sequence $\{E_r, d_r\}$ is a \emph{spectral sequence of $\Box_D$-Hopf algebras}  if for every $r\geq 1$, $(E_r^{*,*}, d_r)$ is a differential bigraded $\Box_D$-Hopf algebra with multiplication $\mu_r$ and comultiplication $\triangle_r$, such that $\{E_r, d_r\}$ is a spectral sequence of $\Box_D$-coalgebras and a spectral sequence of $\Box_D$-algebras using this comultiplication and multiplication. Further, the map 
\[
\chi_{r+1}\colon E_{r+1}^{s,t} \to E_{r+1}^{s,t} 
\]
must be the induced map 
\[
\chi_{r+1} = H_*(\chi_r)\colon H_*(E_r^{s,t};k) \to H_*(E_r^{s,t};k).  
\]
\end{defn}

Having established these definitions, we now consider the algebraic structure on the coB\"okstedt spectral sequence. We first prove that it is a spectral sequence of $\Box$-coalgebras.
\begin{thrm}\label{boxcoalg} Let $C$ be a connected cocommutative coalgebra spectrum, and let $k$ be a field. Then the coB\"okstedt spectral sequence is a spectral sequence of $\Box_{H_*(C;k)}$-coalgebras. \label{sscoalg}
\end{thrm}

\begin{proof}
By definition, the coB\"okstedt spectral sequence for $C$ is the Bousfield--Kan spectral sequence for the cosimplicial object $\coTHH^\bullet(C) \sma Hk$, which we will denote  $X^\bullet$. 

We consider the following commutative diagram of simplicial sets
\[\xymatrix{ S^1_\bullet &S^1_\bullet \vee S^1_\bullet\ar[l]_{\nabla}\\
&S^1_\bullet \amalg S^1_\bullet\ar[u]\ar[ul]^{\nabla'} }
\]
where $\nabla$ and $\nabla'$ are both simplicial fold maps. 
Note that we may identify $S^1_\bullet \vee S^1_\bullet$ with the coequalizer of the two maps $S^1_\bullet \amalg \ast\amalg S^1_\bullet\rightrightarrows S^1_\bullet \amalg S^1_\bullet$
that send the point $\ast$ to the basepoint in either copy of $S^1_\bullet$.

On cosimplicial cotensors, and after smashing with $Hk$, the diagram above yields a diagram of cosimplicial spectra
\[\xymatrix{C^{S^1_\bullet}\sma Hk \ar[r]\ar[dr]& C^{S^1_\bullet \vee S^1_\bullet}\sma Hk\ar[d]\\
&C^{S^1_\bullet \amalg S^1_\bullet}\sma Hk.}
\]
Furthermore the right vertical map here equalizes the two induced maps
\[ C^{S^1_\bullet \amalg S^1_\bullet}\sma Hk\rightrightarrows C^{S^1_\bullet\amalg \ast\amalg S^1_\bullet}\sma Hk.\]

These maps of cosimplicial spectra induce maps of spectral sequences on the corresponding Bousfield--Kan spectral sequences:  let $E_r^{*,*}$ denote the Bousfield--Kan spectral sequence for $X^\bullet$ (that is, the coB\"okstedt spectral sequence), let $D_r^{*,*}$ denote the spectral sequence for $C^{S^1_\bullet \vee S^1_\bullet}\sma Hk$, let $'\!D_r^{*,*}$ denote the spectral sequence for $C^{S^1_\bullet \amalg S^1_\bullet}\sma Hk$ and let $''\!D_r^{*,*}$ denote the spectral sequence for $C^{S^1_\bullet \amalg \ast \amalg S^1_\bullet}\sma Hk$.   In this notation, we have maps of spectral sequences
\[ \xymatrix{ E_r^{*,*} \ar[r]\ar[dr]& D_r^{*,*}\ar[d]\\
& '\!D_r^{*,*}\ar@<-.5ex>[d]\ar@<.5ex>[d]\\
&''\!D_r^{*,*}
}\]
where the map $D_r^{*,*}\to{} '\!D_r^{*,*}$ equalizes the two lower vertical maps.

It is proved in \cite{BGHSZ} that the Alexander--Whitney map identifies the spectral sequence $'\!D_r^{*,*}$ with the tensor product spectral sequence $E_r^{*,*}\otimes_k E_r^{*,*}$ from the $E_2$-page on, using Bousfield and Kan's work on pairings in Bousfield--Kan spectral sequences \cite{BKquadrant,BKpairings}.   In fact, the identifications of Bousfield and Kan  produce natural maps of spectral sequences
\[ '\!D_r^{*,*} \to E_r^{*,*}\otimes_k E_r^{*,*}\text{\quad and\quad } ''\!D_r^{*,*}\to E_r^{*,*}\otimes H_*(C) \otimes E_r^{*,*}.\]
On the 1-pages, the maps in this direction are given by the shuffle map: this is the quasi-inverse of the map $\AW\colon E_1^{*,*}\otimes_k E_1^{*,*} \to {}'\!D_1^{*,*}$ used in \cite{BGHSZ}.

Thus, on the $1$-pages, we have an induced map 
\[ \xymatrix{ D_1^{*,*}\ar[d]\ar@{-->}[r]& E_1^{*,*}\Box_{H_*(C;k)} E_1^{*,*} \ar[d] \\
'\!D_1^{*,*}\ar@<-.5ex>[d]\ar@<.5ex>[d]\ar[r]^-{\sh}& E_1^{*,*}\otimes_k E_1^{*,*}\ar@<-.5ex>[d]\ar@<.5ex>[d]\\
''\!D_1^{*,*}\ar[r]^-{\sh}& E_1^{*,*}\otimes_k H_*(C;k) \otimes_k E_1^{*,*}}
\]
where $\sh$ denotes the shuffle map of the Eilenberg--Zilber theorem.   Hence the composite of the spectral sequence map $E_1^{*,*}\to D_1^{*,*}$ with the dashed arrow above induces a $\Box_{H_*(C;k)}$-comultiplication structure on $E_1^{*,*}$ which we denote by $\triangle_1$.   The comultiplication $\triangle_1$ satisfies the coLeibniz rule because the composite
\[E_1^{*,*}\to '\!\!D_1^{*,*}\xto{\sh} E_1^{*,*}\otimes_k E_1^{*,*}\]
satisfies the coLeibniz rule and the differential $d_r\Box 1 \pm 1\Box d_r$ on $E_r^{*,*}\Box_{H_*(C;k)}E_r^{*,*}$ is the restriction of the differential $d_r\otimes 1\pm 1\otimes d_r$ to the elements of $E_r^{*,*}\otimes_k E_r^{*,*}$ that are equalized.

The comultiplication on the $E_2$-page is induced similarly.  By the usual calculation of the $1$-page of a Bousfield--Kan spectral sequence, we make the identifications 
\begin{align*}
'\!D_1^{s,*}&=N^s\left(H_*(C;k)^{\otimes_k \bullet}\otimes_k H_*(C;k)^{\otimes_k \bullet}\right)\\
''\!D_1^{s,*}&=N^s\left(H_*(C;k)^{\otimes_k \bullet} \otimes_k H_*(C;k) \otimes_k H_*(C;k)^{\otimes_k \bullet}\right).
\end{align*}
Using these identifications, and the fact that normalization commutes with equalizers, we see that the $D_1$-page is
\[D_1^{s,*}=N^s\big(H_*(C;k)^{\otimes_k \bullet}\Box_{H_*(C;k)}H_*(C;k)^{\otimes_k\bullet}\big).\]
The horizontal shuffle maps $\sh$ from the Eilenberg--Zilber map above induce isomorphisms on homology, so we have a diagram of maps of bigraded $k$-modules
\[ \xymatrix{ D_2^{*,*}\ar[d]\ar@{-->}^-{\phi_2}[r]& E_2^{*,*}\Box_{H_*(C;k)} E_2^{*,*} \ar[d] \\
'\!D_2^{*,*}\ar@<-.5ex>[d]\ar@<.5ex>[d]\ar[r]^-{\cong}& E_2^{*,*}\otimes_k E_2^{*,*}\ar@<-.5ex>[d]\ar@<.5ex>[d]\\
''\!D_2^{*,*}\ar[r]^-{\cong}& E_2^{*,*}\otimes_k H_*(C;k) \otimes_k E_2^{*,*}}
\]
The induced dotted map is the map in the $\Box_{H_*(C;k)}$-K\"unneth Theorem, \cref{thrm:boxkunneth}.  We may thus define the comultiplication $\triangle_2$ on $E_2$ to be the composite of the map $E_2^{*,*}\to D_2^{*,*}$ and this induced dotted map; this clearly satisfies the condition of \cref{defn:boxcoalgss}. Again, since the middle map is a map of spectral sequences and the differential on $E_2^{*,*}\Box_{H_*(C;k)}E^{*,*}_2$ is restricted from that on $E_2^{*,*}\otimes_k E_2^{*,*}$, this comultiplication also satisfies the coLeibniz rule.

Because the middle and lower horizontal maps above are maps of spectral sequences, repeated application of the K\"unneth theorem for $k$-modules gives a similar diagram on the $r$-pages for $r\geq 2$; the K\"unneth theorem for $\Box_{H_*(C;k)}$-comodules then induces the desired comultiplications $\triangle_r$ at each level.
\end{proof}

\begin{rem} In the proof of \cref{sscoalg}, we use Bousfield and Kan's shuffle pairing $'\!D_1^{*,*}\to E_1^{*,*}\otimes_k E^{*,*}_1$, which---as remarked---is not the comparison map $\AW$ between these spectral sequences used in defining the coalgebra structure in \cite{BGHSZ}. Instead, the map $\AW$ is the quasi-inverse to Bousfield and Kan's pairing, which is constructed from the K\"unneth isomorphism  and shuffle/Alexander--Whitney maps.  This requires working over a field. Since the comparison map in \cref{sscoalg} and the one of \cite{BGHSZ} are quasi-inverses at the $E_1$-page, we conclude that the $\Box_{H_*(C;k)}$-coalgebra structure on $E_r^{*,*}$ of \cref{sscoalg} is the restriction of the $k$-coalgebra structure produced in \cite{BGHSZ}.
\end{rem}

Simplicial maps of the circle also induce the product structure on the coB\"okstedt spectral sequence. As in \cref{sec:coHH}, we must use a double circle model of $S^1_\bullet$ to get a pinch map that is indeed simplicial. Recall the double circle model $dS^1_{\bullet}$ defined earlier:
\[
dS^1_\bullet =(\Delta^1\amalg \Delta^1)\amalg_{(\partial\Delta^1\amalg\partial\Delta^1)}\partial\Delta^1.  
\]  
There is a simplicial pinch map
 \[
 \psi\colon dS^1_{\bullet} \to S^1_{\bullet} \vee S^1_{\bullet}
 \]
that collapses $\partial\Delta^1$ to a point. For $D$ a $k$-coalgebra, let $d\coHH^\bullet(D)$ denote the cosimplicial $k$-coalgebra $D^{dS^1_\bullet}$.  To parallel \cref{def:coHH1}, we let $d\coHH_*(D)$ denote the homology of the chain complex (of $k$-modules) $C^*(d\coHH^\bullet(D))$ associated to $d\coHH^\bullet(D)$ under the Dold--Kan correspondence.  We will need the following lemma comparing coHochschild homology defined with the standard simplicial model of the circle to that defined with the double circle model. Recall that $\pi\colon dS^1_{\bullet} \to S^1_{\bullet}$ is the collapse map that takes the second $\Delta^1$ in $dS^1_{\bullet}$ to the basepoint. This lemma is the purely algebraic version of \cref{doublecirclearefine}.

\begin{lemma}\label{lemma:pi}
 Let $D$ be a cocommutative (graded) $k$-coalgebra.  The map $\pi\colon dS^1_\bullet \to S^1_\bullet$ induces an isomorphism of (bi-)graded abelian groups. 
\[\pi\colon \coHH_*(D)\to d\coHH_*(D)\]
\end{lemma}
\begin{proof}
Following the proof of \cref{doublecirclearefine}, we may take cotensors with the defining pushout diagrams of simplicial sets to model $\pi$ as the following map on pullbacks in cocommutative $k$-coalgebras:
\[\xymatrix{
D^{\Delta^1_\bullet} \ar[r] & D^{\partial \Delta^1_\bullet} & D^{\Delta^1_\bullet} \ar[l]\\
D^{\Delta^1_\bullet}\ar@{=}[u] \ar[r]&D^{\partial \Delta^1_\bullet}   \ar@{=}[u] & D^\ast\ar[u]\ar[l]
}
\]
 As in \cref{boxprodforcoalg}, pullback in the category of cocommutative $k$-coalgebras agrees with the definition of $\Box$-product.  The coalgebra $D^{\partial \Delta^1_\bullet}$ is $D\otimes D$; this allows us to identify $\pi$ at the cosimplicial level with the map
\[ D^{\otimes \bullet+2} \Box_{D\otimes D} D \to D^{\otimes \bullet +2}\Box_{D\otimes D}D^{\otimes \bullet +2}.\]
The map $\pi\colon \coHH_*(D)\to d\coHH_*(D)$ is the map on homology induced by the associated map of cochain complexes
\[ C^*(D^{\otimes \bullet+2} \Box_{D\otimes D} D) \to C^*(D^{\otimes \bullet +2}\Box_{D\otimes D}D^{\otimes \bullet +2}).\]
Observe that for each $n$, $D^{\otimes n+2}$ is a $D\otimes D$-comodule and by the Eilenberg--Zilber theorem for $D\otimes D$-comodules, this map is quasi-isomorphic to the map
\[ C^*(D^{\otimes \bullet+2})\Box_{D\otimes D} C^*(D) \to C^*(D^{\otimes \bullet+2})\Box_{D\otimes D} C^*(D^{\otimes \bullet+2}).
\]
Now $C^*(D^{\otimes \bullet +2})$  is an injective resolution of $D$ as a $D\otimes D$-bicomodule;  the map $D\to C^*(D^{\Delta^1_\bullet})$ is the augmentation of this resolution.   By \cite[Proposition 4.1]{EilenbergMoore66} (and also see the discussion at the beginning of \S5 there), this map induces an isomorphism on homology as desired.
\end{proof}

\begin{rem}
Observe that the double circle $dS^1$ used in this section is a different version of the double circle than the model $d'S^1$ used in \cref{sec:coHH}. In particular, the orientation of one of the 1-simplices is reversed. In this section we use the model $dS^1$ so that the flip map will be simplicial. We choose to use the model $d'S^1$ in \cref{sec:coHH} because it simplifies the chain-level formulas. This follows the choices made in \cite{AR} in the dual case. 
\end{rem}

\begin{thrm}\label{thrm:algebrass} Let $C$ be a connected cocommutative coalgebra spectrum.  If for $r\geq 2$, each $E_r^{*,*}(C)$ is coflat over $H_*(C;k)$,  then the coB\"okstedt spectral sequence is a spectral sequence of $\Box_{H_*(C;k)}$-Hopf algebras.
 \end{thrm}

\begin{proof}
The comultiplication is given by \cref{sscoalg} above.  To obtain the multiplication, we use the double circle simplicial model of $S^1$ denoted $dS^1_{\bullet}$. Let $dE_r^{*,*}$ denote the Bousfield-Kan spectral sequence for the cosimplicial object $C^{dS^1_{\bullet}} \sma Hk$. There are natural maps 
\[
E_r^{*,*}\square_{H_*(C;k)} E_r^{*,*} \xleftarrow{\ \phi_r\ } D_r^{*,*} \xto{\ \psi\ } dE_r^{*,*} \xleftarrow{\ \pi\ } E_r^{*,*}.\]
The map $\phi_r$ is the comparison map defined  in the proof of \cref{sscoalg} above. The map $\psi$ is induced by the simplicial pinch map $\psi\colon dS^1_{\bullet} \to S^1_{\bullet}\vee S^1_\bullet,$ and the map $\pi$ is induced by the weak equivalence $dS^1_{\bullet} \to S^1_{\bullet}$ given by collapsing. By \cref{lemma:pi} the map $\pi$ is an isomorphism when $r\geq 2$. 

We use induction to show that the map $\phi_r$ is an isomorphism for $r\geq2$.  We first establish the isomorphism when $r=2$. As in \cref{sscoalg} above, the $D_1$-page can be identified as 
 \[D_1^{s,*}=N^s\left(H_*(C;k)^{\otimes_k \bullet}\Box_{H_*(C;k)}H_*(C;k)^{\otimes_k\bullet}\right) \cong E_1^{s,*} \square_{H_*(C;k)} E_1^{s,*},\]
because the normalization of a cosimplicial object commutes with equalizers, as both are given at each level by limit constructions.

On the 1-page the map $\phi_1$ above is the shuffle map
\[
E_1^{s,*} \square_{H_*(C;k)} E_1^{s,*} \to [E_1^{*,*} \square_{H_*(C;k)} E_1^{*,*}]_{s,*}.
\]
which by \cref{prop:EZ} induces an isomorphism on homology:
\[
H_*(D_1^{*,*};k) \to H_*( E_1^{*,*}\square_{H_*(C;k)} E_1^{*,*};k).
\] 
The left hand side is $D_2^{*,*}$. To calculate the right hand side, we can replace $E_1^{*,*}$, the normalized chain complex $N^*\left(H_*(C;k)^{\otimes_k \bullet}\right)$,  by the quasi-isomorphic unnormalized complex $H_*(C;k)^{\otimes *}$, which is coflat over $H_*(C;k)$.  Since $E_2^{*,*} = H_*(E_1^{*,*};k)$ is coflat over $H_*(C;k)$ by hypothesis, \cref{thrm:boxkunneth} implies the right hand side is 
\[H_*( E_1^{*,*}; k) \square_{H_*(C;k)} H_*(E_1^{*,*};k),\]
and hence $\phi_2=H_*(\phi_1)$ is an isomorphism
\[
D_2^{*,*} \to E_2^{*,*}\square_{H_*(C;k)} E_2^{*,*}.
\]
 
Now assume that $\phi_r$ is an isomorphism for some fixed $r\geq 2$. The isomorphism
\[D_r^{*,*} \xto{\ \phi_r\ }E_r^{*,*}\square_{H_*(C;k)} E_r^{*,*}\]
induces an isomorphism on homology.
By hypothesis, $E_{r+1}^{*,*}$ is coflat over $H_*(C;k)$, so \cref{thrm:boxkunneth} applies and the map $\phi_{r+1} = H_*(\phi_r)$ is an isomorphism
\[D_{r+1}^{*,*} \xto{\ \cong\ }  E_{r+1}^{*,*}\square_{H_*(C;k)} E_{r+1}^{*,*}.
\]
Thus by induction the map $\phi_r$ is an isomorphism for all $r\geq 2.$ 

The multiplication 
\[E_r^{*,*}\square_{H_*(C;k)} E_r^{*,*} \xto{\ \mu_r\ }  E_r^{*,*}\]
is then given by $\mu_r = \pi^{-1}\psi\phi_r^{-1}.$
Since the maps $\phi_r$ constructed in the proof of \cref{sscoalg} satisfy the coLeibniz rule and the remaining maps used in constructing $\mu_r$ are maps of spectral sequences, $\mu_r$ satisfies the Leibniz rule.

Finally, we define the antipode map $\chi_r\colon E_r^{*,*}\to E_r^{*,*}$ to be the map of spectral sequences 
\[ \chi\colon E_r^{*,*}\xto{\ \pi\ } dE_r^{*,*} \xto{\ \chi'\ } dE_r^{*,*}\xto{\ \pi^\inv\ } E_r^{*,*}.\]
The middle map $\chi'$ is the map of spectral sequences  induced by the flip map on the double circle and the proof that this satisfies the required conditions for an antipode is analogous to the proof of \cref{prop:suspspectrum_boxcoalgspectrallevel}.  Again, we use that $\pi$ is an isomorphism for $r\geq 2$.
\end{proof}

\section{Computational results}\label{sect:computations}

As discussed in \cref{sec:freeloop}, when $X$ is a simply connected space, $\coTHH(\Sigma_+^{\infty} X)$ can be identified with the suspension spectrum of the free loop space on $X$,
\[
\Sigma^\infty_+ \freeloops X \xrightarrow{\simeq}  \coTHH(\Sigma_+^{\infty} X).
\]
The coB\"okstedt spectral sequence thus provides a method for computing the homology of the free loop space \cite[Corollary 4.5]{BGHSZ}. 

\begin{prop}[\cite{BGHSZ}]\label{cor.conv.2}
{Let $X$ be a simply connected space.  If for each $s$ there is an $r$ such that $E_r^{s, s+i} = E_{\infty}^{s, s+i}$, the coB\"okstedt spectral sequence arising from the coalgebra $\Sigma_+^{\infty} X$ converges completely to 
\[H_{*}(\coTHH(\Sigma_+^{\infty} X);k) \cong H_*(\freeloops X; k).\]
}
\end{prop}

In \cite{BGHSZ} the authors use the coB\"okstedt spectral sequence to compute the homology of free loop spaces $H_*(\freeloops X; k)$ for  instance when $X$ is $\mathbb{C}P^{\infty}, BU(n), BSU(n), BSp(n)$, or products of these. These calculations use the $k$-coalgebra structure on the coB\"okstedt spectral sequence.

In this section we use the $\Box$-Hopf algebra structure on the coB\"okstedt spectral sequence that we produced in \cref{sect:thespectralsequence} to carry out further computations of the homology of free loop spaces.  In particular, we consider spaces with exterior cohomology.  The (co)homology of free loop spaces of simply connected spaces with mod $p$ exterior cohomology has been considered, for instance, in \cite{KY97}, \cite{Kuribayashi11}, \cite{Smith84}, and \cite{KMN14}. Our approach yields new results, as consequences of the following general collapse result for the coB\"okstedt spectral sequence, which we prove later in this section.

\begin{thm}\label{exterior}
Let $k$ be a field of characteristic $p$ and let $C$ be a cocommutative coalgebra spectrum whose homology coalgebra is 
\[
H_*(C;k) = \Lambda_k(y_{i_1}, y_{i_2}, \dots, y_{i_n}).
\]
Here the $y_{i_j}$ are cogenerators in odd degrees, $|y_{i_j}| = i_j,$ and  $i_{j+1} \geq i_j\geq 3$. Then if  $\frac{i_n -2+ \sum_{j=1}^n i_j}{i_1-1}<p$, the coB\"okstedt spectral sequence for $\coTHH(C)$ collapses at $E_2$, and
\[
E_2 \cong E_{\infty} \cong \Lambda_{k}(y_{i_1}, y_{i_2}, \dots, y_{i_n}) \otimes k[w_{i_1}, w_{i_2}, \dots, w_{i_n}],
\] 
with $y_{i_j}$ in bidegree $(0, i_j)$ and $w_{i_j}$ in bidegree $(1, i_j)$. 
\end{thm}

Letting $C= \Sigma_+^{\infty}X$, for $X$ a simply connected space with exterior cohomology, \cref{exterior} yields results on the homology of free loop spaces. We first consider spaces $X$ whose cohomology is exterior on two generators. This case has been considered previously in work of Kuribayashi and Yamaguchi \cite{KY97}, and we compare our results to that previous work in \cref{RemarkKY}. The following result follows directly from \cref{exterior}.
\begin{thm}\label{exteriorloop2}
Let $k$ be a field of characteristic $p$ and let $X$ be a simply connected space whose cohomology is exterior on two generators in odd degrees,
\[
H^*(X; k) \cong  \Lambda_k(x_{i_1}, x_{i_2}),
\]
$|x_{i_j}|=i_j$, and  $i_1 \leq i_{2} \leq \frac{p-1}{2} i_1 - \frac{p-1}{2}$. Then the homology of the free loop space on $X$ is given as a graded $k$-module by
\[
H_*(\freeloops X; k) \cong \Lambda_{k}(y_{i_1}, y_{i_2}) \otimes k[w_{i_1}, w_{i_2}],
\]
where $|y_{i_j}| = i_j$, and $|w_{i_j}| = i_j - 1$. 
\end{thm}

\begin{rem}\label{RemarkKY}
Kuribayashi and Yamaguchi \cite{KY97} compute using different methods the cohomology of free loop spaces $H^*(\freeloops X;\mathbb{Z}/p)$ for $X$ simply connected with mod $p$ cohomology isomorphic to $\Lambda(x_{i_1}, x_{i_2})$, where $i_1 \leq i_2 \leq 2i_1 - 2$, and $p>3$. Dualizing their result yields $H_*(\freeloops X; \mathbb{Z}/p) \cong \Lambda(y_{i_1}, y_{i_2}) \otimes k[w_{i_1}, w_{i_2}]$ with the generators $y_i$ and $w_i$ in the degrees indicated in the theorem above. When $p=5$ our statement applies in the same range of degrees as the Kuribayashi and Yamaguchi result. For $p>5$, though, our statement applies in a much broader range than the statement in \cite{KY97}, and provides an extension of their result. Note, however, that in the current work we compute $H_*(\freeloops X;k)$ as a $k$-module. In  \cite{KY97}  they compute the algebra structure on $H^*(\freeloops X; \mathbb{Z}/p)$, in addition to the module structure. We expect that the isomorphism in \cref{exteriorloop2} holds as an isomorphism of coalgebras, but this will be addressed in subsequent work.   \end{rem}

As noted above, our work greatly expands the range in which we can understand the homology of the free loops on a space whose mod $p$ cohomology is exterior on two generators.  We next consider spaces whose cohomology is exterior on more than two generators. When $p=2$, this was studied in work of Smith \cite{Smith84}. For $p>2$, while the homology of free loop spaces for spaces with exterior cohomology with one or two generators had been studied in past work, previous techniques did not easily extend beyond the case of two generators. One advantage of the new approach presented here is that it yields results in much greater generality. Indeed, using our new approach we are able to prove the following general result for spaces with exterior cohomology with $n$ generators.
\begin{thm}\label{exteriorloop}
Let $k$ be a field of characteristic $p$ and let $X$ be a simply connected space whose cohomology is exterior on a finite number of generators
\[
H^*(X; k) \cong  \Lambda_k(x_{i_1}, x_{i_2}, \dots, x_{i_n}),
\]
where the $x_{i_j}$ are generators in odd degrees, $|x_{i_j}|=i_j$, and  $i_{j+1} \geq i_j$. Then when $\frac{i_n + \sum_{j=1}^n i_j}{i_1-1} \leq p$, the homology of the free loop space on $X$ is given as a graded $k$-module by
\[
H_*(\freeloops X; k) \cong \Lambda_{k}(y_{i_1}, y_{i_2}, \dots, y_{i_n}) \otimes k[w_{i_1}, w_{i_2}, \dots, w_{i_n}],
\]
where $|y_{i_j}| = i_j$, and $|w_{i_j}| = i_j - 1$. 
\end{thm}

\begin{ex}
For appropriate choices of $p$, \cref{exteriorloop} recovers calculations of the homology of free loop spaces of $SU(n), Sp(n), G_2, F_4, E_6, E_7,$ and $E_8$. 
\end{ex}
In the case where the cohomology of $X$ is exterior on one generator in odd degree, our techniques yield a stronger result. We consider this case in \cref{exterior1} later in this section.

The approach to proving the above results is to use the coB\"okstedt spectral sequence for $\coTHH(\Sigma_+^{\infty}X)$ and exploit the additional algebraic structure developed in \cref{sect:thespectralsequence} for the coB\"okstedt spectral sequence. The following proposition will be very useful. 

\begin{prop}\label{prop:differentials}
Let $C$ be a cocommutative coalgebra spectrum such that $H_*(C;k)$ is connected and $\coHH(H_*(C;k))$ is coflat over $H_*(C;k)$. Then the $E_2$-term of the coB\"okstedt spectral sequence for $\coTHH(C)$, 
\[
E_2^{*,*}(C) = \coHH_*(H_*(C; k)),
\]
is a $\Box_{H_*(C; k)}$-bialgebra, and the shortest non-zero differential $d_r^{s,t}$ in lowest total degree $s+t$, maps from a $\Box_{H_*(C;k)}$-algebra indecomposable to a $\Box_{H_*(C;k)}$-coalgebra primitive. 
\end{prop}

\begin{proof}
As seen in \cref{coHHbialg}, $\coHH_*(H_*(C;k))$ is a $\boxbialg{H_*(C;k)}$.  Suppose $d_1 = d_2 = \dots= d_{r-1} = 0.$ We then consider what happens on $E_r^{*,*}(C)$. Note that $E_r^{*,*}(C) = E_2^{*,*}(C)$ is still a $\boxbialg{H_*(C)}$, and the differential $d_r$ satisfies both the Leibniz rule:
\[
d_r \circ \mu = \mu \circ (d_r \Box id \pm id \Box d_r), 
\]
and the coLeibniz rule:
\[
\triangle \circ d_r = (d_r \Box id \pm id \Box d_r) \circ \triangle.
\]
Here $\Box$ denotes $\Box_{H_*(C;k)}$. Suppose $xy$ is decomposable and $d_r(xy) \neq 0$. Then by the Leibniz formula one of $d_r(x)$ or $d_r(y)$ is nonzero, so there is a class in lower total degree with a nonzero differential.  Now consider any $w \in E_r^{*,*}(C) $. We can write the comultiplication on $w$ as
\[
\triangle(w) = w \Box 1 + 1 \Box w + {\textstyle\sum_i} w_i^{(1)} \Box w_i^{(2)}.
\]
Suppose $d_r(w)$ is not primitive. Then the comultiplication on $d_r(w)$ can be written
\[
\triangle(d_r(w)) = d_rw \Box 1 + 1 \Box d_rw + {\textstyle\sum_i} (d_rw)_i^{(1)} \Box (d_rw)_i^{(2)},
\]
where at least one of the terms $(d_rw)_i^{(1)} \Box (d_rw)_i^{(2)}$ is nonzero. By the coLeibniz rule it then follows that at least one of the terms $d_r(w_i^{(1)})$ or $d_r(w_i^{(2)})$ must be nonzero. Therefore there is a class in lower total degree with a nonzero differential. 
\end{proof}

We now have the pieces in place to prove \cref{exterior}.

\begin{proof}[Proof of \cref{exterior}]
We consider the coB\"okstedt spectral sequence computing $H_*(\coTHH(C);k)$. This spectral sequence has $E_2$-term
\[
E_2 \cong \coHH_{*}(H_*(C;k)) \cong \coHH_*(\Lambda_k(y_{i_1}, y_{i_2}, \dots, y_{i_n})).
\]
This coHochschild homology was computed in \cite{BGHSZ}, Proposition 5.1, so we have
\[
E_2 \cong \coHH_{*}(H_*(C;k)) \cong \Lambda_{k}(y_{i_1}, y_{i_2}, \dots, y_{i_n}) \otimes k[w_{i_1}, w_{i_2}, \dots, w_{i_n}],
\] 
where $w_j$ is in degree $(1, j)$. This coalgebra is coflat over the exterior algebra $\Lambda_{k}(y_{i_1}, y_{i_2}, \dots, y_{i_n})\cong H_*(C;k)$, so \cref{prop:differentials} applies, and the shortest non-zero differential, from lowest total degree, maps from a $\Box_{H_*(C;k)}$-indecomposable to a $\Box_{H_*(C;k)}$-primitive. The $\Box_{H_*(C;k)}$-coalgebra structure that \cref{coHHbialg} gives on the $E_2$-page $H_*(C;k) \otimes k[w_{i_1}, w_{i_2}, \dots, w_{i_n}]$ agrees with the one given in \cref{prop:primitive}, and hence \cref{prop:primitive} 
describes the $\Box_{H_*(C;k)}$-primitive elements in $\coHH_{*}(H_*(C;k))$.  However, from \cite{BGHSZ}, we know that the coB\"okstedt spectral sequence is also a spectral sequence of $k$-coalgebras. It follows that the shortest non-zero differential, in lowest total degree, maps to a $k$-coalgebra primitive. We can see from \cref{prop:primitive} that there are fewer primitive elements of $\coHH_*(H_*(C;k))$ viewed as a $k$-coalgebra than viewed as a $\Box_{H_*(C;k)}$-coalgebra. Let $p$ denote the characteristic of $k$. Then the primitive elements of $\coHH_*(H_*(C;k))$ as a $k$-coalgebra are the elements of the form $y_j \otimes 1$ and $1 \otimes w_j^{p^m}$. The elements $y_j \otimes 1$ are in bidegree $(0, j)$ and hence cannot be hit by a differential. Similarly, the elements $1 \otimes w_j$ are in bidegree $(1, j)$ and cannot be hit by a differential.  So the first non-zero differential, if one exists, has to hit some $1 \otimes w_j^{p^m}$, for $m\geq 1$.

We argue that for large enough values of $p$, these classes cannot be hit by a differential from a $\Box_{H_*(C;k)}$-indecomposable. In other words, if the characteristic of $k$ is large enough, this spectral sequence will collapse. In \cref{prop:indecomposable} below we establish that the $\Box_{H_*(C;k)}$-indecomposable elements of $\coHH_*(H_*(C;k))$ are the elements of the form $x \otimes w_i$, where $x\in \Lambda_{k}(y_{i_1}, y_{i_2}, \dots, y_{i_n})$. Consider an indecomposable element $y_{i_{j_1}}\!\!\dotsm y_{i_{j_m}}\otimes w_{i_j}$ and suppose $d_r(y_{i_{j_1}}\!\!\dotsm y_{i_{j_m}}\otimes w_{i_j} )$ is $1 \otimes w_{i_a}^{p^b}$. Note that the bidegree of $d_r(y_{i_{j_1}}\!\!\dotsm y_{i_{j_m}}\otimes w_{i_j} )$ is $(1+r, i_{j_1} + \dots+ i_{j_m} + i_j +r-1)$.  By comparing bidegrees, 
\[
1+r = p^b \text{\qquad and \qquad}
i_{j_1} + \dots+ i_{j_m} + i_j +r-1 = i_ap^b.
\]
So in particular,
\[
i_{j_1} + \dots+ i_{j_m} + i_j -2 =  (i_a-1)p^b.
\]
Since $i_a \neq 1$ it follows that 
\[
p^b = \frac{i_{j_1} + \dots +i_{j_m} + i_j -2}{(i_a-1)}.
\]
Therefore 
\[
p \leq p^b = \frac{i_{j_1} + \dots +i_{j_m} + i_j -2}{(i_a-1)} \leq \frac{i_n -2+ \sum_{j=1}^n i_j}{i_1-1}.
\]
\end{proof}

To complete the proof of \cref{exterior} we prove the following lemma, which identifies the $\Box_{H_*(C;k)}$- indecomposable elements.

\begin{prop}\label{prop:indecomposable}
For a field $k$, the indecomposable elements of the $\Box_{\Lambda_{k}(y_1, y_2, \ldots y_n)}$-algebra 
\[\coHH_*(\Lambda_{k}(y_1, y_2, \ldots y_n)) \cong \Lambda_{k}(y_1, y_2, \dots, y_n) \otimes k[w_1, w_2, \dots, w_n]
\] are of the form $x \otimes w_i$ for any $x\in \Lambda_{k}(y_1, y_2, \dots, y_n)$. 
\end{prop}

\begin{proof}
Let $C$ denote the coalgebra $\Lambda_{k}(y_1, y_2, \ldots, y_n)$.  Let $D$ denote the coalgebra $k[w_1, w_2, \ldots ,w_n]$. Recall from \cref{coHHbialg} the $\Box_{C}$-algebra structure on $\coHH_*(C) \cong C \otimes D$. In order to identify the indecomposable elements in this $\Box_{C}$-algebra, we first need to identify the augmentation ideal. The augmentation 
\[
\epsilon\colon C \otimes D \to C
\]
is given by the composite
\[
C \otimes D \xto{\id \otimes \epsilon_D}  C \otimes k \xto{\ \cong\ }  C,
\]
where $\epsilon_D$ is the counit for $D$. This is the map $\epsilon_D\colon k[w_1, w_2, \dots, w_n] \to k$ that sends all of the $w_i$ to zero. Let $\overline{D}$ denote the kernel of $\epsilon_D$. Since we are working over a field $k$, the augmentation ideal $I(C\otimes D)$ is then given by $I(C\otimes D) = C \otimes \overline{D}$. The indecomposable elements $Q(C\otimes D)$ are determined by the exact sequence
\[
I(C\otimes D) \Box_C I(C\otimes D) \xto{\ \mu\ }  I(C\otimes D) \longrightarrow Q(C\otimes D) \longrightarrow 0.
\]
Here we view the map $\mu$ as a composite 
\[
I(C\otimes D) \Box_C I(C\otimes D) \xhookrightarrow{\quad} 
  (C\otimes D) \Box_C (C\otimes D) \xto{\mu_{C\otimes D}}  I(C\otimes D), 
\]
where $\mu_{C\otimes D}$ is the product as in \cref{coHHbialg}. We rewrite the exact sequence as 
\[
(C \otimes \overline{D}) \Box_C (C \otimes \overline{D}) \xto{\ \mu\ } C \otimes \overline{D} \longrightarrow Q(C\otimes D) \longrightarrow  0.
\]
Note that 
\[(C \otimes \overline{D}) \Box_C (C \otimes \overline{D})  \cong C \otimes \overline{D} \otimes \overline{D}.
\]
This isomorphism is induced by the comultiplication followed by a twist:
\[C\otimes \overline{D}\otimes \overline{D}\xto{\triangle\otimes 1\otimes 1} C\otimes C\otimes \overline{D}\otimes \overline{D}\xto{\id \otimes \tau\otimes \id} C \otimes \overline{D} \otimes C\otimes \overline{D}.\]
The box product $(C\otimes \overline{D})\Box_C(C\otimes \overline{D})$ is identified with the image of this composite inside the 4-fold tensor product.   It follows from the definition of the $\Box_C$-multiplication in \cref{coHHbialg} that under this isomorphism the multiplication $\mu$ is given by 
\[\mu\colon C \otimes \overline{D} \otimes \overline{D} \xto{\id \otimes \mu_D} C \otimes \overline{D},
\]
where $\mu_D$ denotes the multiplication on the classes $w_i^j$ described in \cref{prop:exteriorbialg}. The cokernel of this map is then given by
\[ Q(C\otimes D)=C\otimes \coker(\mu_D),\] 
that is, all elements of the form $x\otimes w_i$.
\end{proof}

In the case where the cohomology is exterior on only one generator, we have the stronger result that the spectral sequence always collapses, regardless of the characteristic of $k$. This result was also proven in \cite{KY97} using a different approach. 
\begin{thm}\label{exterior1}
Let $X$ be a simply connected space whose cohomology is exterior on one generator in odd degree
\[
H^*(X; k) \cong  \Lambda_k(x).
\]
Then the homology of the free loop space on $X$ is given as a graded $k$-module by
\[
H_*(\freeloops X; k) \cong \Lambda_{k}(y) \otimes k[w],
\]
where $|y| = |x|$, and $|w| = |x| - 1$. 
\end{thm}
\begin{proof} This follows directly from the proof of \cref{exterior}. In the case of one exterior generator in degree $i$, the formulas in the proof of \cref{exterior} show that in order for a nonzero differential $d_r$, $r\geq 2$, to exist in the spectral sequence, the following formulas need to be satisfied:  
\[
p^b \leq \frac{2i - 2}{i-1} = 2 \hspace{1cm} \textup{and} \hspace {1cm} 1+r = p^b.
\]
These formulas imply that $r<2$, so there is no such differential and the spectral sequence collapses. 
\end{proof}

\begin{rem}
\cref{exterior1} computes in particular the homology of free loop spaces of spheres, $H_*(\mathcal{L}S^n; k)$, for $n>1,$ odd. The homology of free loop spaces of spheres was computed classically by Ziller \cite{Ziller77}, using Morse theory, and we recover some of those results here. 
\end{rem}
\begin{rem}
A result using similar proof techniques to \cref{exterior1} is proven by Klanderman in \cite{Klanderman20}. Klanderman uses the relative coB\"okstedt spectral sequence to compute $\pi_*(\coTHH^{Hk}(C))$ for $C$ a cocommutative $Hk$-coalgebra spectrum with $\pi_*(C) \cong \Lambda_k(y)$, for $|y|>1$, odd.  
\end{rem}

\bibliography{coTHH}
\bibliographystyle{plain}

\end{document}